\documentclass[reqno,10pt]{amsart}

\numberwithin{equation}{section}

\usepackage{verbatim} % hacer comentarios
\usepackage{xspace} % dejar espacios
\usepackage{amssymb}
\usepackage{amsmath}
\usepackage{amsthm} %
\usepackage[latin1]{inputenc}
\usepackage{graphicx}
\usepackage{amscd,amssymb,euscript,graphics}
\makeindex
\newcommand{\Z}{\ensuremath{\mathbb{Z}}}
\newcommand{\Q}{\ensuremath{\mathbb{Q}}}
\newcommand{\R}{\ensuremath{\mathbb{R}}}
\newcommand{\C}{\ensuremath{\mathbb{C}}}
\newcommand{\F}{\ensuremath{\mathbb{F}}}
\newcommand{\N}{\ensuremath{\mathbb{N}}}

\newtheorem{teo}{Theorem}[section]
\newtheorem{prop}[teo]{Proposition}
\newtheorem{cor}[teo]{Corollary}

\newtheorem{thm}{Theorem}
\theoremstyle{remark}
\newtheorem{claim}[thm]{Claim}
\theoremstyle{definition}

\newtheorem{example}[thm]{Example}

\newtheorem{df}[teo]{Definition}

\newtheorem{lemma}[teo]{Lemma}
\newtheorem{conj}[teo]{Conjecture}
\newtheorem{remark}[teo]{Remark}

\newcommand{\bc}{\begin{center}}
\newcommand{\ec}{\end{center}}
\newcommand{\be}{\begin{enumerate}}
\newcommand{\ee}{\end{enumerate}}
\newcommand{\bi}{\begin{itemize}}
\newcommand{\ei}{\end{itemize}}
\newcommand{\bd}{\begin{description}}
\newcommand{\ed}{\end{description}}
\newcommand{\beq}{\begin{equation}}
\newcommand{\eeq}{\end{equation}}
\newcommand{\beqa}{\begin{eqnarray}}
\newcommand{\eeqa}{\end{eqnarray}}
\newcommand{\bfr}{\begin{flushright}}
\newcommand{\efr}{\end{flushright}}
\newcommand{\bfl}{\begin{flushleft}}
\newcommand{\efl}{\end{flushleft}}

\newcommand{\fq}
{\mathbb{F}_q}

\newcommand{\fqk}
{\mathbb{F}_{q^k}}

\newcommand{\fqm}
{\mathbb{F}_q^{\times}}

\newcommand{\fqmc}
{\widehat{\mathbb{F}_{q}^{\times}}}

\newcommand{\fqkm}
{\mathbb{F}_{q^k}^{\times}}

\newcommand{\fqkmc}
{\widehat{\mathbb{F}_{q^k}^{\times}}}

\newcount\theTime
\newcount\theHour
\newcount\theMinute
\newcount\theMinuteTens
\newcount\theScratch
\theTime=\number\time
\theHour=\theTime
\divide\theHour by 60
\theScratch=\theHour
\multiply\theScratch by 60
\theMinute=\theTime
\advance\theMinute by -\theScratch
\theMinuteTens=\theMinute
\divide\theMinuteTens by 10
\theScratch=\theMinuteTens
\multiply\theScratch by 10
\advance\theMinute by -\theScratch

\def\today{{\number\day\space
 \ifcase\month\or
  January\or February\or March\or April\or May\or June\or
  July\or August\or September\or October\or November\or December\fi
 \space\number\year}}

\begin{document}

\title{Hypergeometric Functions over Finite Fields and their relations to Algebraic Curves}
\author{M. Valentina Vega} 
%\thanks{
 % G. H. Tucci
 % is with Bell Laboratories,
 % Alcatel--Lucent, 600 Mountain Ave, Murray Hill, NJ 07974.
 % E-mail: gabriel.tucci@alcatel-lucent.com}

\begin{abstract}
In this work we present an explicit relation between the number of points on a family of algebraic curves over $\F_{q}$ and sums of values of certain hypergeometric functions over $\F_{q}$. Moreover, we show that these hypergeometric functions can be explicitly related to the roots of the zeta function of the curve over $\F_{q}$ in some particular cases.  A general conjecture relating these last two is presented and advances toward its proof are shown in the last section.
\end{abstract}

\maketitle

%\begin{IEEEkeywords}
%Random Matrices, Limiting Distribution, Gaussian Averages, MIMO Capacity, MMSE, Covariance Matrices
%\end{IEEEkeywords}

\section{Introduction}%

\vspace{0.3cm}
\par The problem of finding the number of solutions over a finite field of a polynomial equation has been of interest to mathematicians for many 
years. A typical result in this direction is the \emph{Hasse-Weil bound}, which states that a smooth projective curve 
of genus $g$ defined over a finite field with $q$ elements has between $q+1-2g\sqrt{q}$  and $q+1+2g\sqrt{q}$ points. A natural question to ask is whether there are simple formulas for counting points in terms of interesting mathematical objects.

\vspace{0.3cm}
Classical hypergeometric functions and their relations to counting points on curves over finite fields 
have been investigated by mathematicians since the beginnings of 1900. Recall that for $a_1, \dots , a_r, b_1, \dots , b_s$, $x \in \mathbb{C}$, the classical hypergeometric series is defined by 
\vspace{0.1cm}
\begin{equation}
 _{r}F_{s} \left( 
\begin{matrix}
a_1, & a_2, & \dots, & a_r \\
b_1, & b_2, & \dots, & b_s \\
\end{matrix}
\Bigg{\vert} x
\right) 
:= \sum_{k=0}^{\infty}\frac{(a_1)_k (a_2)_k \cdots (a_r)_k}{(b_1)_k (b_2)_k \cdots (b_s)_k }\frac{x^k}{k!}
\end{equation}
where $(a)_k:=a(a+1)\cdots (a+k-1)$ is the Pochhammer symbol.

\vspace{0.3cm}
Many connections between classical hypergeometric series, elliptic curves and modular forms have been discovered. For example, if we consider the Legendre family of elliptic curves given by $y^2=x(x-1)(x-t)$, $t\neq 0,1$, and denote
$$_2F_1[a,b;c|t]:={}_{2}F_{1} \left( 
\begin{matrix}
a, & b \\
   & c \\
\end{matrix}
\Bigg{\vert} t
\right),$$ the specialization 
$_2F_1[\frac{1}{2},\frac{1}{2};1|t]$ is a multiple of an elliptic integral which represents a period of the lattice 
associated to the previous family, as Kummer showed. 
%in 1836. 
For another examples, Beukers \cite{Be93} related a 
period of $y^2=x^3-x-t$ to the values  $_{2}F_{1}[\frac{1}{12},\frac{5}{12};\frac{1}{2}|\frac{27}{4}t^2]$.

\vspace{0.3cm}
\par In the 1980's, J. Greene \cite{Gr84,Gr87} initiated a study of hypergeometric series over finite fields. 
Let $q$ be a power of a prime, and let $\widehat{\mathbb{F}_{q}^{\times}}$ denote the group of multiplicative characters $\chi$ on $\F_{q}^{\times}$, extended to all of $\F_q$ by setting $\chi(0)=0$.  If $A, B \in  \widehat{\mathbb{F}_{q}^{\times}}$ we let the binomial coefficient be a Jacobi sum. Specifically,  
$$\binom{A}{B}:=\frac{B(-1)}{q}J(A,\overline{B})=\frac{B(-1)}{q} \sum_{x \in \widehat{\mathbb{F}_{q}}} A(x)\overline{B}(1-x)$$  
In this notation, we recall Greene's definition of hypergeometric functions over $\F_q$.
If $A_0,A_1,\dots,A_n,$ and $B_1,B_2,\dots,B_n$ are characters of $\widehat{\mathbb{F}_{q}^{\times}}$ and $x\in\F_q$, then the \emph{Gaussian hypergeometric function over $\F_{q}$} is defined by
\begin{equation}\label{hypergeometric function}
 _{n+1}F_{n} \left( 
\begin{matrix}
A_0, & A_1, & \dots, & A_n \\
     & B_1, & \dots, & B_n \\
\end{matrix}
\Bigg{\vert} x
\right) 
:= \frac{q}{q-1} \sum_{\chi\in\widehat{\mathbb{F}_{q}^{\times}}} \binom{A_0\chi}{\chi} \binom{A_1\chi}{B_1\chi} \dots \binom{A_n\chi}{B_n\chi} \chi(x)
\end{equation}
where $n$ is a positive integer.  
%(See Chapter \ref{preliminaries}, Section \ref{hpgf over fq} for more details.)

\par Greene explored the properties of these functions and found that they satisfy many summation and transformation formulas analogous to 
those satisfied by the classical functions. These similarities generated interest in finding connections that hypergeometric functions over finite fields may have with other objects, for example elliptic curves. In recent years, many results have been proved in this direction and as expected, certain families of elliptic curves are closely related to particular hypergeometric functions over finite fields. Motivated by these types of results, we have explored more relations between Gaussian hypergeometric functions and counting points on varieties over finite fields.

\vspace{0.3cm}
\par Throughout, let $\F_q$ denote the finite field with $q$ elements, where $q$ is some prime power.  For $z \in \F_{q}$ let $\mathcal{C}_{z}$ be the smooth projective curve with affine equation \begin{equation}\label{curve_ms}
\mathcal{C}_{z}: y^l=t^m(1-t)^s(1-zt)^m
\end{equation}
where $l \in \N$  and $1\leq m,s<l$ such that $m+s=l$.
% and denote by $\# \mathcal{C}_{z} (\F_{q})$ the number of points the curve has over $\F_{q}$. 
Our first result provides an explicit relation between the number of points on certain family of curves over finite fields and values of particular hypergeometric functions. 

\begin{teo} \label{thmformula}
Let $a=m/n$ and $b=s/r$ be rational numbers such that $0<a,b<1$, and let $z \in \F_{q}$, $z \neq 0,1$. Consider the smooth projective algebraic curve with affine equation given by
$$\mathcal{C}_{z}^{(a,b)}: y^l = t^{l(1-b)}(1-t)^{lb}(1-zt)^{la}$$ 
where $l:=\text{lcm}(n,r)$.
If $q \equiv 1 \pmod{l}$ then:
\begin{equation}\label{formula}
\# \mathcal{C}_{z}^{(a,b)}(\F_{q})= q+1+q\sum_{i=1}^{l-1} \eta_{q}^{ilb}(-1) \,{}_{2}F_{1} \left(
\begin{array}{ll|}
\eta_{q} ^{il(1-a)}, & \eta_{q} ^{il(1-b)} \\
     & \varepsilon \\
\end{array} \: z \right)
\end{equation}
where $\eta_{q} \in \widehat{\F_{q}^{\times}}$ is a character of order $l$, and $\# \mathcal{C}_{z}^{(a,b)}(\F_{q})$ denotes the number of points that the curve $\mathcal{C}_{z}^{(a,b)}$ has over $\F_q$.
\end{teo}

\vspace{0.4cm}
\par After recalling, in section \ref{recent history}, a few recent results that relate counting points on varieties over $\F_p$ to hypergeometric functions, we set up the necessary preliminaries in section \ref{preliminaries} and give the proof of Theorem \ref{thmformula} and some consequences of it in section \ref{hgf and ac}. Our next interest has been to find a closed formula for hypergeometric functions over finite fields, and more specifically, we have been interested in relating each particular term that appears in the right hand side of sum (\ref {formula}) to the curve $\mathcal{C}_{z}^{(a,b)}$.  Explicitly, in section \ref{mainconj} we state a conjecture that relates the values of  the hypergeometric functions appearing in (\ref {formula}) to counting points on the curves $\mathcal{C}_{z}^{(a,b)}$ over $\F_q$, and in sections \ref{proof l=3} and \ref{proof l=5} we prove this conjecture for some particular cases. These results give a closed formula for the values of hypergeometric functions over finite fields in terms of the traces of Frobenius of certain curves. Finally, in section \ref{generalconj} we show advances toward the proof of the conjecture in its full generality.
%we give a closed formula for the values of hypergeometric functions over finite fields in terms of the traces of Frobenius of certain curves.  Based on these results and many %computational examples done with Magma, in Section \ref{conjecture} we state a conjecture that relates  the values of  the hypergeometric functions appearing in (\ref {formula}) to %counting points on the curves $\mathcal{C}_{z}^{(a,b)}$ over $\F_q$. In this section we also show advances toward the proof of the conjecture in its full generality.\\
\vspace{0.3cm}
 
\section{Recent History}\label{recent history}

\vspace{0.3cm}
\par As mentioned at the beginning of this paper, following Greene's introduction of hypergeometric functions over $\F_q$ in the 1980s,
results emerged linking their values to counting points on varieties over $\F_q$. 

\vspace{0.3cm}

Consider the two families of elliptic curves over $\F_p$ defined by
\begin{align*}
E_1(t)&:y^2=x(x-1)(x-t),\,\,\,\, t\neq 0,1\\
E_2(t)&:y^2=(x-1)(x^2+t),\,\,\,\, t \neq 0,-1.
\end{align*}
Then, for $p$ and odd prime define the traces of Frobenius on the above families by

\begin{align}\label{traces Frobenius}
a_1(p,t)&=p+1- \#E_1(t)(\F_p) \nonumber\\
a_2(p,t)&=p+1- \#E_2(t)(\F_p)
\end{align}
where, for i=1,2
 $$\#E_{i}(t)(\F_{p}):=\#\{(x,y)\in E_{i}(t): x,y\in \F_{p}\}\cup \{P\}$$ 
denotes the number of points the curve $E_{i}(t)$ has over the finite field $\F_{p}$, with $P=[0:1:0]$ being the point at infinity.
Denote by $\phi$ and $\varepsilon$ the quadratic and trivial characters on $\F_{p}^{\times}$ respectively.
% i.e., for $a \in \F_{p}^{\times}$
%$$\phi(a)=
%\begin{cases}
%1 & \textnormal{if} \,\, x^{2}=a \,\,\text{is solvable in} \,\,\F_{p}\\
%-1 & \textnormal{if} \,\, x^{2}=a\,\,\text{is not solvable in} \,\,\F_{p}
%\end{cases}$$ 
%is the Legendre symbol, and 
%$$\varepsilon(a)=1.$$
Then, the families of elliptic curves defined above are closely related to particular hypergeometric functions over $\F_p$.  For example, $_2F_1[\phi,\phi;\varepsilon|t]$ arises in the formula for Fourier coefficients of a modular form associated to $E_1(t)$ \cite{Ko92,On98}.  Further, Koike and Ono, respectively, gave the following explicit relationships:

\begin{teo}[(1) Koike \cite{Ko92}, (2) Ono \cite{On98}] \label{koike_ono} Let $p$ be an odd prime. Then
\begin{enumerate}
\item  for $t \neq 0,1$:
%For the family of elliptic curves given by $E_1(t):y^2=x(x-1)(x-t)$, $t\neq 0,1$, then
$$p \, _{2}F_{1}\left(
\begin{matrix}
\phi,&\phi\\
     &\varepsilon
\end{matrix}
\Bigg{\vert} t
\right) =-\phi(-1) a_1(p,t)$$\\
%where $a_1(p,t) = p+1- \#E_1(t)(\F_p)$.
\item for $t \neq 0,-1$: %For the family of elliptic curves given by $E_2(t):y^2=(x-1)(x^2+t)$, $t\neq 0,-1$, then
$$p^2 \, 
 _{3}F_{2} \left( 
\begin{matrix}
\phi, & \phi, & \phi \\
     & \varepsilon, & \varepsilon \\
\end{matrix}
\Bigg{\vert} 1+\frac{1}{t}
\right) 
=
\phi(-t)(a_2(p,t)^2-p).$$
%where $a_2(p,t) = p+1- \#E_2(t)(\fp)$.
\end{enumerate}
\end{teo}

\vspace{0.3cm}
In addition, Frechette, Ono, and Papanikolas \cite{FOP04} gave explicit relations between the number of points over $\F_{p}$ in the more general varieties defined by
%counting points on more general varieties over $\F_{p}$ and hypergeometric functions over finite fields. Specifically, for $p$ an odd prime, $k\geq 4$ even, and the varieties defined by
%define three sequences of varieties $\mathcal{U}_{k}$, $\mathcal{V}_{k}$, and $\mathcal{W}_{k}$ by
\begin{align*}
\mathcal{U}_{k}&: y^2=\prod_{i=1}^{k-2}(x_{i}-1)(x_{i}^2+t),\\
\mathcal{V}_{k} &: y^2=\prod_{i=1}^{k-2}x_{i}(x_{i}-1)(x_{i}-t),\\
\mathcal{W}_{k} &: y^2=\prod_{i=1}^{k-2}x_{i}(x_{i}-1)(x_{i}-t^2)
\end{align*}
and the traces of Frobenius defined in (\ref{traces Frobenius}). (A different approach and other applications of these hypergeometric functions can be found in \cite{Ka90})

%Then, the number of points in $\mathcal{U}_{k}(\F_{p})$, $\mathcal{V}_{k}(\F_{p})$ and $\mathcal{W}_{k}(\F_{p})$ are directly related to values of certain hypergeometric functions over $\F_{p}$. In fact, they are related to the number of points in $E_{1}(\F_{p})$ and $E_{2}(\F_{p})$ which, by Theorem \ref{koike_ono}, are related to the hypergeometric functions. Specifically, they showed that:
%\begin{align*}
%\#\mathcal{U}_{k}(\F_{p})&=p^{k-1}+2+\sum_{t=1}^{p-2} a_{2}(p,t)^{k-2},\\
%\#\mathcal{V}_{k}(\F_{p})&=p^{k-1}+2+\sum_{t=2}^{p-1} a_{1}(p,t)^{k-2},\\
%\#\mathcal{W}_{k}(\F_{p})&=p^{k-1}+3+\sum_{t=2}^{p-1} (1+\phi(t))\,a_{1}(p,t)^{k-2}.
%\end{align*}\\

\section{{Preliminaries on multiplicative characters, hypergeometric functions and the Zeta function of a variety}\label{preliminaries}}

\vspace{0.3cm}
\par In this section we fix some notation and recall a few facts regarding multiplicative characters, hypergeometric functions and the Zeta function of a variety that will be needed in later sections.
%\vspace{0.3cm}
%\subsection{Multiplicative Characters}
%\vspace{0.3cm}
\par Let $p$ be a prime and let $\F_{q}$ be a finite field with $q$ elements, with $q=p^r$ for some positive integer $r$. We will denote by $\F_{q}^{\times}$ the multiplicative group of $\F_{q}$, i.e., $\F_{q}^{\times}=\F_{q}-\{0\}$. Now we state the \emph{orthogonality relations} for multiplicative characters, of which we will make use in section \ref{hgf and ac}. 

\begin{lemma}
 Let $\chi$ be a multiplicative character on $\fqm$. Then

$(a) \quad \displaystyle{\sum_{x\in\fq}\chi(x)=
\begin{cases}
q-1 & \textnormal{if} \,\, \chi=\varepsilon\\
0 & \textnormal{if} \,\, \chi\neq\varepsilon
\end{cases}}$

$(b) \quad \displaystyle{\sum_{\chi \in \fqmc}\chi(x)=
\begin{cases}
q-1 &  \textnormal{if} \,\, x=1\\
0 & \textnormal{if} \,\, x \neq 1.
\end{cases}}$
\end{lemma}

We also require a few properties of hypergeometric functions over $\F_p$ that Greene proved in \cite{Gr87}.
Greene defined  the \emph{Gaussian hypergeometric functions over $\fq$} as the following character sum:

\begin{df}[\cite{Gr87} Defn. 3.5] \label{gauss_hpgf}
For characters $A, B, C \in \widehat{\mathbb{F}_{q}^{\times}}$ and $x \in \fq$

\vspace{0.1cm}
\begin{equation}
_{2}F_{1} \left( 
\begin{matrix}
A, & B \\
   & C \\
\end{matrix}
\Bigg{\vert} x
\right) 
:= \varepsilon(x)\frac{BC(-1)}{q} \sum_{y \in \fq}B(y)\overline{B}C(1-y)\overline{A}(1-xy).
\end{equation} 
\end{df}

\vspace{0.3cm}

More generally, Greene proved the following theorem which connects these functions to Jacobi sums, and 
extended the previous definition to a higher number of multiplicative characters (see formula (\ref{hypergeometric function})).

\begin{teo}[\cite{Gr87} Theorem 3.6]
For characters $A, B, C \in \widehat{\mathbb{F}_{q}^{\times}} $ and $x \in \fq$,
\vspace{0.1cm}
$$_{2}F_{1} \left( 
\begin{matrix}
A, & B \\
   & C \\
\end{matrix}
\Bigg{\vert} x
\right) 
=\frac{q}{q-1}\sum_{\chi \in \fqmc}\binom{A\chi}{\chi}\binom{B\chi}{C\chi}\chi(x).$$
\end{teo}

A comprehensive introduction to these functions can be found in Greene's paper \cite{Gr87}, where he presented many properties and transformation identities they satisfy. One transformation that is of interest to us is presented in the next theorem, and it allows to replace the arguments $A, B \in \fqmc$ by $\overline{A}, \overline{B}$ respectively.

\begin{teo}[\cite{Gr87} Theorem 4.4]
 If $A,B,C \in \fqmc$ and $x \in \fq$, then
\begin{align}\label{conjugate_hyperg_function}
 _{2}F_{1} \left( 
\begin{matrix}
A, & B \\
   & C \\
\end{matrix}
\Bigg{\vert} x
\right) 
&= C(-1)C\,\overline{AB}(1-x)\,_{2}F_{1} \left( 
\begin{matrix}
C\overline{A}, & C\overline{B} \\
   & C \\
\end{matrix}
\Bigg{\vert} x
\right) \nonumber\\
& \qquad + A(-1)\binom{B}{\overline{A}C}\delta(1-x)
\end{align}
where $\delta(x)= 
\begin{cases}
 1 &  \textnormal{if} \,\, x=0\\
0 & \textnormal{if} \,\, x \neq 0.
\end{cases}
$
\end{teo}

In particular, when $A$ and $B$ are inverses of each other and $C=\varepsilon$ we get the following result.

\begin{cor}\label{conjugate_hpgf}
 Let $A \in \fqmc$ and $x \in \fq\backslash \{1\}$. Then
$$_{2}F_{1} \left( 
\begin{matrix}
A, & \overline{A} \\
   & \varepsilon \\
\end{matrix}
\Bigg{\vert} x
\right) 
={}_{2}F_{1} \left( 
\begin{matrix}
\overline{A}, & A \\
   & \varepsilon \\
\end{matrix}
\Bigg{\vert} x
\right) $$
\begin{proof}
 Just notice that, since $x \neq 1$ then the last term in the right hand side of (\ref{conjugate_hyperg_function}) vanishes, and $A\overline{A}(1-x)=1$.
\end{proof}
\end{cor}

Now, recall that the \emph{Zeta function of a projective variety} is a generating function for the number of solutions of a set of polynomial equations defined over a finite field $\F_{q}$, in finite extension fields $\F_{q^n}$ of $\F_{q}$. In this way, we collect all the information about counting points into a single object. 

\vspace{0.3cm}

%Again, let $p$ be an odd prime and let $\fq$ denote the finite field with $q$ elements where $q=p^r$ for some positive integer $r$.
%Let $\mathcal{V}$ be a projective variety, so $\mathcal{V}$ is the zero-set 
%$$f_1(x_{0},\dots,x_{N})=\cdots = f_{m}(x_{0}, \dots ,x_{N})=0$$
%of a collection of homogeneous polynomials with coefficients in $\mathbb{F}_{q}$. Denote by 
%$\mathcal{V}(\mathbb{F}_{q^{n}})$ the set of points of $\mathcal{V}$ with coordinates in $\mathbb{F}_{q^{n}}$, where 
%$\mathbb{F}_{q^{n}}$ is the field extension of degree $n$ of $\fq$.

\begin{df}
Let $\mathcal{V}$ be a projective variety. The zeta function of $\mathcal{V}/\fq$ is the power series
$$Z(\mathcal{V}/\fq; T):= \text{exp} \left( \sum_{n=1}^{\infty} \# \mathcal{V}(\mathbb{F}_{q^n})\frac{T^n}{n}
\right) \in \mathbb{Q}[[T]].$$
\end{df}
%\noindent (Here if $F(T) \in \Q [[T]]$ is a power series with no constant term, then exp$(F(T))$ is the power series $\sum_{i=0}^{\infty}F(T)^{i}/i!$). Thus, the zeta function %$Z(\mathcal{V}/\fq; T)$ associated to $\mathcal{V}$ contains all the information concerning the number of points of $\mathcal{V}$ over each field extension of $\fq$ of finite %degree. 

%\begin{obs} Notice that, once we know $Z(\mathcal{V}/\fq; T)$, it is not hard to recover the numbers $\#\mathcal{V}(\mathbb{F}_{q^n})$ by the formula
%$$\#\mathcal{V}(\mathbb{F}_{q^n})=\frac{1}{(n-1)!}\frac{d^n}{dT^n}\log Z(\mathcal{V}/\fq;T)
%\bigg\arrowvert _{T=0}.$$
%\end{obs}

\vspace{0.3cm}
 In 1949, Andr\'e Weil \cite {We49} made a series of conjectures concerning the number of points on 
varieties defined over finite fields which were proved by Weil, Dwork \cite {Dw60} and Deligne \cite{De74} in later years.
 %[see \cite{Fu69}]. 
Applying these conjectures to a smooth projective curve $\mathcal{V}$ of genus $g$ defined over $\fq$, we obtain that
\begin{equation}\label{zetafunction}
Z(\mathcal{V}/\fq;T)=\frac{(1-\alpha_{1}T)(1-\overline{\alpha_{1}}T)\cdots(1-\alpha_{g}T)
(1-\overline{\alpha_{g}}T)}{(1-T)(1-qT)}
\end{equation}
\vspace{0.2cm}
where $\alpha_{i} \in \C, \,\, |\alpha_{i}|=\sqrt{q}$ for all $i=1,\dots, g$. 
%Denote $a_{i}:=\alpha_{i}+\bar{\alpha_{i}}$ for $1\leq i \leq g$.
In this case we have a beautiful formula for counting points on $\mathcal{V}$ over $\mathbb{F}_{q^n}$, namely
\begin{equation}\label{formulapuntos}
\#\mathcal{V}(\mathbb{F}_{q^n})=q^n+1-\sum_{i=1}^{g}(\alpha_{i}^n+\overline{\alpha_{i}}^n)
\end{equation}

\noindent (For details see \cite{IR90}) We will make strong use of formulas (\ref{zetafunction}) and (\ref{formulapuntos}) applied to particular families of curves to prove the results in the following sections.\\
\vspace{0.3cm}

\section{Counting Points on Families of Curves over Finite Fields}\label{hgf and ac}
\vspace{0.4cm}

\par 
We consider the problem of connecting the number of points that certain families of curves have over finite fields to values of particular hypergeometric functions over finite fields. Throughout, let $\fq$ denote the finite field with $q$ elements, where $q$ is some prime power. We start with a result that allows to count the number of solutions of a particular equation by using multiplicative characters on $\fq$.

\begin{lemma}\label{countingwithsum}
Let $q$ be a prime and $a \in \fq\backslash \{0\}$. If $n | (q-1)$ then
$$\#\{x \in \F_{q} : x^n = a \} = \sum_{\chi ^n =\varepsilon} \chi(a)$$
where the sum runs over all characters $\chi \in \fqmc$ of order dividing $n$.

\begin{proof}
We start by seeing that there are exactly $n$ characters of order dividing $n$. Let $\chi : \fqm \to \C^{\times}$ be a character such that $\chi^{n}=\varepsilon$ and let $g \in \fqm$ be a generator. Since $\chi ^n = \varepsilon$, the value of $\chi(g)$ must be an $n$th root of unity, hence there are at most $n$ such characters. Consider $\chi \in \fqmc$ defined by $\chi(g)=e^{2\pi i/n}$ (i.e. $\chi(g^k)=e^{2\pi ik/n}$).  It is easy to see that $\chi$ is a character and $\varepsilon, \chi, \chi^2, \cdots , \chi^{n-1}$ are $n$ distinct characters of order dividing $n$. Therefore, there are exactly $n$ characters of order dividing $n$.

\vspace{0.1in}
Now let $a \neq 0$ and suppose that $x^n=a$ is solvable; i.e., there is an element $b \in \F_{q}$ such that $b^n=a$. Since $\chi^n=\varepsilon$ we have that $\chi(a)=\chi(b^n)=\chi(b)^n=1$. Thus $$\sum_{\chi^n=\varepsilon} \chi(a)=\sum_{\chi^n=\varepsilon} 1 =n$$ 
Also notice that in this case, $\#\{x \in \F_{q} : x^n = a \}=n$ because if $x^n =a \pmod{q}$ is solvable then there exist exactly gcd$(n,\varphi(q))$ solutions, where $\varphi$ denotes the Euler function. But since $\varphi(q)=q-1$ and $n|(q-1)$ it follows that gcd$(n,q-1)=n$ (for a proof of this result see \cite{IR90} Proposition 4.2.1).

\vspace{0.1in}
To finish the proof we need to consider the case when $x^n=a$ is not solvable, in which case $\#\{x \in \F_{q} : x^n = a \}=0$.
Call $T:=\sum_{\chi^n=\varepsilon} \chi(a)$. Since $x^n =a$ is not solvable, there exist a character $\rho$ such that $\rho ^n=\varepsilon$ and $\rho (a)\neq 1$ (take $\rho (g)=e^{2\pi i/n}$ where $\langle g\rangle=\fqm$). Since the characters of order dividing $n$ form a group, it follows that $\rho (a)T=T$. Then $(\rho (a)-1)T=0$ which implies that $T=0$ since $\rho (a)\neq 1$.
\end{proof}
\end{lemma}

\vspace{0.1cm}

We are now ready to prove Theorem \ref{thmformula}.
\vspace{0.3cm}
%Similar to the results given in Section \ref{intro}, the following theorem provides an explicit relation between the number of points on certain family of curves over finite fields and %values of particular hypergeometric functions. 

%%%%%%%%%%%%%%%%%%%%%%%%%%%%%%%%%%%%%%%%%%%%%%
%%formula for counting points on curve%%%%%%%%
%%%%%%%%%%%%%%%%%%%%%%%%%%%%%%%%%%%%%%%%%%%%%%

%\begin{teo} \label{thmformula}
%Let $a=m/n$ and $b=s/r$ be rational numbers such that $0<a,b<1$, and let $z \in \F_{q}$, $z \neq 0,1$. Consider the smooth projective algebraic curve with affine equation given by
%$$\mathcal{C}_{z}^{(a,b)}: y^l = t^{l(1-b)}(1-t)^{lb}(1-zt)^{la}$$ 
%where $l:=\text{lcm}(n,r)$.
%If $q \equiv 1 \pmod{l}$ then:
%\begin{equation}\label{formula}
%\# \mathcal{C}_{z}^{(a,b)}(\F_{q})= q+1+q\sum_{i=1}^{l-1} \eta_{q}^{ilb}(-1) \,{}_{2}F_{1} \left(
%\begin{array}{ll|}
%\eta_{q} ^{il(1-a)}, & \eta_{q} ^{il(1-b)} \\
 %    & \varepsilon \\
%\end{array} \: z \right)
%\end{equation}
%where $\eta_{q} \in \fqmc$ is a character of order $l$, and $\# \mathcal{C}_{z}^{(a,b)}(\F_{q})$ denotes the number of points that the curve $\mathcal{C}_{z}^{(a,b)}$ has over $\fq$.

%%%%%%%%%%%%%%%%%%%%%%%%%%%%%%%%%%
%%%%% PRUEBA THEOREMA FORMULA %%%%%%%
%%%%%%%%%%%%%%%%%%%%%%%%%%%%%%%%%

\begin{proof}[Proof of Theorem \ref{thmformula}]
To simplify the notation, we will denote the curve $\mathcal{C}_{z}^{(a,b)}=\mathcal{C}_{z}$. Since $\fqmc$ is a cyclic group of order $q-1$ and $l | (q-1)$ there exists a character $\eta_{q} \in \fqmc$ of order $l$. 
Recall that $\mathcal{C}_{z}$ is a projective curve, so adding the point at infinity we have
$$\# \mathcal{C}_{z}(\F_{q})  = 1+ \sum_{t \in \F_{q}}\# \{ y \in \F_{q} : y^l = t^{l(1-b)}(1-t)^{lb}(1-zt)^{la} \}$$

\noindent Breaking the sum and applying Lemma \ref{countingwithsum} we see that:
\allowdisplaybreaks{
\begin{align}\label{suma}
\# \mathcal{C}_{z}(\F_{q}) & = 1+ \sum_{\substack {t \in \F_{q} \\ t^{l(1-b)}(1-t)^{lb}(1-zt)^{la}\neq 0}} \# \{ y \in \F_{q} : y^l = t^{l(1-b)}(1-t)^{lb}(1-zt)^{la} \}\nonumber\\  
&\qquad + \# \{ t \in \F_{q} : t^{l(1-b)}(1-t)^{lb}(1-zt)^{la}=0 \}\nonumber\\
  & = 1+ \sum_{t \in \F_{q}} \sum_{i=0}^{l-1} \eta_{q}^i (t^{l(1-b)}(1-t)^{lb}(1-zt)^{la}) \hspace{1.5in} \text{(Lemma \ref{countingwithsum})}\nonumber\\ 
  &\qquad + \# \{ t \in \F_{q} : t^{l(1-b)}(1-t)^{lb}(1-zt)^{la}=0 \}.\nonumber
\end{align}
}
\noindent Now, by separating the sum according to whether $i=0$, and collecting the second and last terms into a single one we have
\begin{align}
 \#\mathcal{C}_{z}(\F_{q}) & = 1+ \sum_{t \in \F_{q}} \varepsilon(t^{l(1-b)}(1-t)^{lb}(1-zt)^{la}) + \sum_{t \in \F_{q}} \sum_{i=1}^{l-1} \eta_{q}^i (t^{l(1-b)}(1-t)^{lb}(1-zt)^{la})\nonumber\\
  &\qquad + \# \{ t \in \F_{q} : t^{l(1-b)}(1-t)^{lb}(1-zt)^{la}=0 \}\nonumber\\
  & = 1+ q + \sum_{t \in \F_{q}} \sum_{i=1}^{l-1} \eta_{q}^i (t^{l(1-b)}(1-t)^{lb}(1-zt)^{la})\nonumber \\
%  & = 1+ q + \sum_{t \in \F_{q}} \sum_{i=1}^{l-1}\eta_{q}^{il(1-b)}(t)\eta_{q}^{ilb}(1-t)\eta_{q}^{la}(1-zt)\nonumber\\
  & = 1+ q + \sum_{i=1}^{l-1} \sum_{t \in \fq}\eta_{q}^{il(1-b)}(t)\,\eta_{q}^{ilb}(1-t)\,\eta_{q}^{la}(1-zt).
\end{align}
The last equality follows from the multiplicativity of $\eta_{q}$ and switching the order of summation.

\vspace{0.3cm}

\noindent On the other hand, by Definition \ref{gauss_hpgf} in section \ref{preliminaries}, we have 
\allowdisplaybreaks{
\begin{align}\label{hypergeometric}
q\,{}_{2}F_{1} \left(
\begin{array}{cc|}
\eta^{il(1-a)}, & \eta^{il(1-b)} \\
     & \varepsilon \\
\end{array} \: z \right)& =
\varepsilon(z) \eta^{il(1-b)}(-1)\sum_{t\in \F_{q}} \eta^{il(1-b)}(t)\,\overline{\eta^{il(1-b)}}(1-t)\overline{\eta^{il(1-a)}}(1-zt)\nonumber \\
& = \varepsilon(z) \eta^{il(1-b)}(-1)\sum_{t\in \F_{q}} \eta^{il(1-b)}(t)\,\eta^{ilb}(1-t)\,\eta^{ila}(1-zt).
\end{align}
}
\noindent Since $z \neq 0$, combining (\ref{suma}) and (\ref{hypergeometric}) we get the desired result.
%\begin{flushright}
%$\Box$
%\end{flushright}

\end{proof}
%\end{teo}
%%%%%%%%%%%%%%%%%%%%%%%%%%%%%%%%%%%%%%%%%%
%%%%%%%%%%%%% FIN DE LA PRUEBA %%%%%%%%%%%%%%%%
%%%%%%%%%%%%%%%%%%%%%%%%%%%%%%%%%%%%%%%%%%%

\vspace{0.1cm}

In the proof of Theorem \ref{thmformula} we applied Lemma \ref{countingwithsum} which requires for $q$ to be a prime number in a particular congruence class modulo $l$. However, Theorem \ref{thmformula} is valid over any finite field extension $\fqk$ of $\fq$ as we see in the next Corollary.

\begin{cor}\label{generalthm}
With same notation as in Theorem \ref{thmformula}, we have that 
$$\# \mathcal{C}_{z}^{(a,b)}(\F_{q^k})= q^k+1+q^k\sum_{i=1}^{l-1} \eta_{q^k}^{ilb}(-1)\,{}_{2}F_{1} \left(
\begin{array}{ll|}
\eta_{q^k} ^{il(1-a)}, & \eta_{q^k} ^{il(1-b)} \\
     & \varepsilon \\
\end{array} \: z \right)$$
where $\eta_{q^k} \in \fqkmc$ is a character of order $l$.

\begin{proof}
Again, denote the curve by $\mathcal{C}_{z}$. First notice that $\fqkmc$ is a cyclic group of order $q^k-1$. Then, if $l|(q-1)$ it also divides $q^k-1$, hence there exists $\eta_{q^k} \in \fqkmc$ of order $l$.

%\vspace{0.2cm}

Next, we show that Lemma \ref{countingwithsum} is also true over $\fqk$ for any positive integer $k$. The proof is almost identical. We only need to check that if $a \in \fqkm$ and $x^n=a$ is solvable, then $\#\{x \in \fqk : x^n=a\}=n$.
For this recall the following two statements, one of which was already used in the proof of Lemma \ref{countingwithsum} (for proofs of them see \cite{IR90} Propositions 4.2.1 and 4.2.3):
\begin{enumerate}
 \item If $(a,q)=1$, then $x^n\equiv a \pmod{q}$ is solvable $\iff$ $a^{\varphi(q)/d}\equiv 1 \pmod{q}$, where $d:=\text{gcd}(n,\varphi(q))$. Moreover, if a solution exists then there are exactly $d$ solutions.
 \item Let $q$ be an odd prime such that $q \nmid a$ and $q \nmid n$. If $x^n \equiv a \pmod{q}$ is solvable, then $x^n\equiv a \pmod{q^k}$ is also solvable for all $k \geq 1$. Moreover all these congruence have the same number of solutions.
\end{enumerate}

\noindent Then, for $q$ prime and in the case $x^n =a$ is solvable we have $$\#\{x \in \F_{q^k} : x^n=a\}=\# \{x \in \F_{q} : x^n=a\}=\text{gcd}(n,\varphi(q))=\text{gcd}(n,q-1)=n$$ since $n|(q-1)$.
Hence, Lemma \ref{countingwithsum} generalizes over $\F_{q^k}$. The proof of the Corollary now follows analogously to the proof of Theorem \ref{thmformula}.

\end{proof}

\end{cor}

\vspace{0.2cm}

As a consequence of Corollary \ref{generalthm} we get the following result that relates the number of points of certain curves over finite extensions of $\fq$.

\begin{cor}\label{curves_have_same_number_points}
Let $l$ be a prime, $m,m',s,s'$ be integers satisfying $1\leq m,m',s,s'< l$ and $m+s=m'+s'=l$, and consider the curves with affine equations given by $\mathcal{C}_{z}^{(m,s)}: y^l=t^m(1-t)^s(1-zt)^m$ and 
$\mathcal{C}_{z}^{(m',s')}:y^l=t^{m'} (1-t)^{s'} (1-zt)^{m'}$ with $z \neq 0,1$. Then, for a prime $q$ such that $q \equiv 1 \pmod{l}$  we have 
$$\#\mathcal{C}_{z}^{(m,s)}(\mathbb{F}_{q^k})=\#\mathcal{C}_{z}^{(m',s')}(\mathbb{F}_{q^k})$$
for all $k \in \mathbb{N}$.
 
\begin{proof}
Again, we drop the dependency of the curves on the integers $m,m',s,s'$ and denote $\mathcal{C}_{z}^{(m,s)}=\mathcal{C}_{z}$ and $\mathcal{C}_{z}^{(m',s')}=\mathcal{C}'_{z}$.
Let $\eta_{q^k} \in \fqkmc$ be a character of order $l$. If $l=2$ then $\mathcal{C}_{z}=\mathcal{C}'_{z}$ since $(m,s)$ and $(m',s')$ are both $(1,1)$. Therefore, there is nothing to prove in this case. 

\vspace{0.3cm}
Suppose now that $l$ is an odd prime. Then, the order of $\eta_{q^k}$ is odd and so $\eta_{q^k}(-1)=1$. Next, consider $a:=m/l,\,\, b:=s/l$ and $a':=m'/l,\,\, b':=s'/l$ in Theorem \ref{thmformula}. The curves defined by these values are exactly $\mathcal{C}_{z}$ and $\mathcal{C'}_{z}$, hence by Corollary \ref{generalthm} and taking into account that $m+s=l$ and $m'+s'=l$, we have

\begin{equation}\label{pointsCz} \#\mathcal{C}_{z}(\F_{q^k})-(q^k+1)= q^k \sum_{i=1}^{l-1}  {}_{2}F_{1} \left(
\begin{array}{ll|}
\eta_{q^k}^{i(l-m)}, & \eta_{q^k}^{im} \\
     & \varepsilon \\
\end{array} \: z \right) \end{equation}

\begin{equation} \label{pointsC'z} \#\mathcal{C}'_{z}(\F_{q^k})-(q^k+1)= q^k \sum_{i=1}^{l-1} {}_{2}F_{1} \left(
\begin{array}{ll|}
\eta_{q^k}^{i(l-m')}, & \eta_{q^k}^{im'} \\
     & \varepsilon \\
\end{array} \: z \right)\end{equation}
%where $\eta_{q^k} \in \widehat{F_{q^k}^*}$ is a character of order $l$.
As we can see, the exponents of the characters appearing in the hypergeometric functions in (\ref{pointsCz}) and (\ref{pointsC'z}) add up to $0 \pmod{l}$.
Also notice that 
\begin{itemize}
\item $\# \{(r,t): 1\leq r,t \leq l-1, r+t=l\}=l-1.$
\item $i(l-m)\equiv j(l-m) \pmod{l} \iff im\equiv jm \pmod{l} \iff l| m(i-j).$
Since $l$ is prime and $0<m<l$, $l$ must divide $i-j$. But $1\leq i,j \leq l-1$, then $i(l-m)\equiv j(l-m) \pmod{l} \iff i=j$  
\end{itemize}

\noindent By these two observations, we see that the terms appearing in the RHS of (\ref{pointsCz}) are the same ones appearing in the RHS of (\ref{pointsC'z}), therefore we conclude that $$\#\mathcal{C}_{z}(\F_{q^k})=\#\mathcal{C}'_{z}(\F_{q^k})$$ 
\end{proof}

\end{cor}

It is not hard to see that the previous result can be generalized to the case when $l$ is an odd integer and $(l,m)=(l,m')=1$, and the argument is the same done above. However, the result is not true in general if we just ask for $m+s=m'+s'$, as we can see in the following example for $l=5$ and $m+s=4$:
\begin{itemize}
 \item If $(m,s)=(1,3)$ then $Z(\mathcal{C}_{2}|\F_{11},T)=\frac{(11T^2+3T+1)^4}{(1-T)(1-11T)}$ hence $$|\#\mathcal{C}_{2}(\F_{11})-(11+1)|=12$$ 
\item If $(m',s')=(2,2)$ then $Z(\mathcal{C'}_{2}|\F_{11},T)=\frac{(11T^2-2T+1)^4}{(1-T)(1-11T)}$ hence $$|\#\mathcal{C}'_{2}(\F_{11})-(11+1)|=8$$
\end{itemize}
\vspace{0.3cm}

\section{The genus of $\mathcal{C}_{z}$}\label{genus of the curve}
\vspace{0.4cm}

\par
%Our next interest is to find a closed formula for hypergeometric functions over finite fields. More specifically, we are interested in relating each particular term that appears in the right hand side of sum (\ref {formula}) to the curve $\mathcal{C}_{z}$. 
In this section, we start by recalling the Riemann-Hurwitz genus formula, which is extremely useful when trying to compute the genus of an algebraic curve.
\vspace{0.3cm}

%%%%%%%%%%%%%%%%%%%%%%%%%%%
%%Riemann-Hurwitz formula%%------------de planetMath
%%%%%%%%%%%%%%%%%%%%%%%%%%%
\begin{teo}[Riemann-Hurwitz genus formula]\label{RHgenus}
 Let $\mathcal{C}_{1}$ and $\mathcal{C}_{2}$ be two smooth curves defined over a perfect field $K$ of genus $g_{1}$ and $g_{2}$ respectively. Let $\psi:\mathcal{C}_{1}\to \mathcal{C}_{2}$ be a non-constant and separable map. Then
$$2g_{1}-2\geq deg(\psi)(2g_{2}-2)+\sum_{P\in \mathcal{C}_{1}} (e_{\psi}(P)-1)$$
where $e_{\psi}(P)$ is the ramification index of $\psi$ at $P$. Moreover, there is equality if and only if either char$(K)=0$ or char$(K)=p$ and $p$ does not divide $e_{\psi}(P)$ for all $P\in\mathcal{C}_{1}$.
\end{teo}

Next, we apply the Riemann-Hurwitz formula to compute the genus of the smooth projective curve $\mathcal{C}_{z}$ with affine equation \begin{equation}\label{curve_ms}
\mathcal{C}_{z}: y^l=t^m(1-t)^s(1-zt)^m
\end{equation}
where $l$ is prime and $1\leq m,s<l$ such that $m+s=l$. For that, we consider the map $$\psi: \mathcal{C}_{z}\to \mathbb{P}^{1}, \,\,\,\,\,\,\,\, [x:y:z] \mapsto [x:z]$$
and notice that $[0:1:0]\mapsto [1:0]$. Generically, every point in $\mathbb{P}^{1}$ has $l$ preimages, so the degree of this map is $l$. Now, the genus of $\mathbb{P}^{1}$ is $0$ and $\psi$ is ramified at 4 points, namely $P_{1}=[0:0:1], P_{2}=[1:0:1], P_{3}=[z^{-1}:0:1]$ and $P_{4}=[0:1:0]$ the point at infinity, with ramification indices $e_{\psi}(P_{i})=l$ for all $i=1, \dots,4$.  Denoting $g:=\text{genus}(\mathcal{C}_{z})$, we obtain that $2g-2= -2l+4(l-1)=2l-4$, hence $g=l-1$. 

\vspace{0.3cm}
\begin{remark}
 The fact that the curve $\mathcal{C}_{z}$ has genus $l-1$ can also be seen by noticing that $\mathcal{C}_{z}$ is a hyperelliptic curve and has model $Y^2=F(X)$ with deg$(F(X))=2l$ (see section \ref{generalconj} Theorem \ref{birational_curve_general_case}). Hence, $2l=2\,\text{genus}(\mathcal{C}_{z})+2$, therefore, genus$(\mathcal{C}_{z})=l-1$.  
\end{remark}

\vspace{0.3cm}

Now, applying Theorem \ref {thmformula} to the curve (\ref{curve_ms}), we see that the upper limit in the sum is the genus of the curve. Also, as we mentioned in the previous section, since $l$ is prime and $\eta_{q}\in \fqmc$ is a character of order $l$, we have that $\eta_{q}(-1)=1$. Then,   
\begin{align}\label{puntoscurva}
\# \mathcal{C}_{z}(\F_{q}) &= q+1 + q\sum_{i=1}^{g}{}_{2}F_{1} \left(
\begin{array}{ll|} 
\eta_{q} ^{is}, & \eta_{q} ^{im} \\
     & \varepsilon \\
\end{array} \: z \right)\nonumber\\
& = q+1+F_{1,q}(z)+F_{2,q}(z)+ \cdots +F_{g,q}(z)
\end{align}
where $F_{i,q}(z)= q \,{}_{2}F_{1} \left(
\begin{array}{ll|}
\eta_{q} ^{is}, & \eta_{q} ^{im} \\
     & \varepsilon \\
\end{array} \: z \right)$.
% and $\eta_{q} \in \mathbb{F}_{q}^{\times}$ is a character of order $l$.

\vspace{0.3cm}

Notice the resemblance between formulas (\ref{formulapuntos}) and (\ref{puntoscurva}). With this similarity in mind, we are now interested in finding relations between the terms in these formulas, i.e., relations between the $F_{i,q}(z)$ in formula (\ref{puntoscurva}) and the $\alpha_{i,q}(z)+\overline{\alpha_{i,q}}(z)$ in formula (\ref{formulapuntos}). 
%Denote $a_i(z):=\alpha_{i}(z)+\overline{\alpha_{i}}(z)$, for $1\leq i \leq g$. We state now our main conjecture, which we proved in some particular cases.
\vspace{0.3cm}

\section{The Main Conjecture}\label{mainconj} 
\vspace{0.4cm}

%As we mentioned in the previous section, we want to relate the terms $F_{i,q}(z)$ in formula (\ref{puntoscurva}) to the $\alpha_{i,q}(z)+\overline{\alpha_{i,q}}(z)$ in formula (\ref{formulapuntos}). 
In this section we state our main conjecture, which proposes an equality between values of ${}_{2}F_{1}$- hypergeometric functions and reciprocal roots of the zeta function of
$\mathcal{C}_{z}$, and in the next sections we prove the conjecture in some particular cases. Denote $a_{i,q}(z):=\alpha_{i,q}(z)+\overline{\alpha_{i,q}}(z)$. 

\begin{conj}\label{conject}
%(Vega \cite{Ve08}). 
Let $l$ and $q$ be odd primes such that 
$q\equiv 1 \pmod{l}$ and let 
$z \in \F_{q}$, $z\neq 0,1$. 
Consider the smooth projective curve with affine equation given by 
$$\mathcal{C}_{z}^{(m,s)}: y^l=t^m(1-t)^s(1-zt)^m$$ 
where $1\leq m,s< l$ are integers such that $m+s=l$. Then, using the notation from previous section and after rearranging terms if necessary
$$F_{i,q}(z)=-a_{i,q}(z) \,\,\,\, \text{for all}\,\,\, 1\leq i\leq g.$$
\end{conj}

The previous conjecture gives a closed formula for the values of some hypergeometric functions over finite fields in terms of the traces of Frobenius of certain curves. 
%In the next two sections we prove the conjecture for particular values of the prime $l$. 
\vspace{0.3cm}

\section{Proof of Conjecture for $l=3$}\label{proof l=3}
\vspace{0.3cm}

\par Throughout this section fix $l=3$ and let $q$ be a prime such that $q \equiv 1 \pmod{3}$. Let $z \in \F_{q}$, $z\neq 0,1$, and consider the smooth projective curve with affine equation given by
\begin{equation}\label{curve_l=3}
 \mathcal{C}_{z}^{(1,2)}: y^3=t(1-t)^{2}(1-zt)
\end{equation}

\noindent Denote $\mathcal{C}_{z}^{(1,2)}=\mathcal{C}_{z}$. The idea will be to show that the $L$-polynomial of the curve $\mathcal{C}_{z}$ is a perfect square, and from that and formulas (\ref{formulapuntos}) and (\ref{puntoscurva}) conclude that the values of the traces of Frobenius must agree with the values of the hypergeometric functions, up to a sign.

\vspace{0.2cm}
\par Recall that, by the Riemann-Hurwitz formula, $\mathcal{C}_{z}$ has genus 2.
Now, every curve of genus 2 defined over $\F_{q}$ is birationally equivalent over $\F_{q}$ to a curve of the form
\begin{equation}\label{canonical_form_genus_2}
\mathcal{C}: Y^2=F(X) 
\end{equation}
where $$F(X)=f_{0}+f_{1}X+f_{2}X^2+ \dots + f_{6}X^6 \in \F_{q}[X]$$ is of degree 6 and has no multiple factors (see \cite{Cass96}). This identification is unique up to a fractional linear transformation of $X$, and associated transformation of $Y$,
\begin{equation}\label{frac_linear_transf}
X\to \frac{aX+b}{cX+d}, \,\,\,\,\, Y\to \frac{eY}{(cX+d)^3}
 \end{equation}
where 
$$a,b,c,d \in \F_{q}, \,\,\,\, ad-bc\neq 0, \,\,\,\, e\in \F_{q}^{\times}.$$
In our particular case we have

\begin{lemma}\label{curve_biratl}
The curve $\mathcal{C}_{z}: y^3=t(1-t)^{2}(1-zt)$ is birationally equivalent to
\begin{equation}\label{curve_bir_equiv}
\mathcal{C}: Y^2=X^6+2(1-2z)X^3+1.
\end{equation}

\begin{proof}
We begin by translating $t\to 1-t$, so the double point is now at the origin. We get:
\begin{align*}\mathcal{C}_{(1)}: y^3& = (1-t)t^2(1-z(1-t))\\
& =(1-z)t^2+(2z-1)t^3-zt^4.
\end{align*}

\noindent Since $z\neq 0$, multiply both sides by $z^{-1}$ and define 
$$G_{2}(t,y):= (1-z^{-1})t^2$$
$$G_{3}(t,y):=z^{-1}y^3-(2-z^{-1})t^3$$
$$G_{4}(t,y):=t^4.$$ Then, each $G_{i}$ is a homogeneous polynomial of degree $i$ in $\F_{q}[t,y]$ and $\mathcal{C}_{z}$ is birationally equivalent to $$\mathcal{C}_{(1)}: G_{2}(t,y)+G_{3}(t,y)+G_{4}(t,y)=0.$$

\noindent Next, put $y=tX$ and complete the square to get:
\allowdisplaybreaks{
\begin{align*}
 \mathcal{C}_{(2)} & : 0=t^4+(z^{-1}X^3+z^{-1}-2)t^3+(1-z^{-1})t^2 \\
&\qquad = \left(t^2+\frac{1}{2}(z^{-1}X^3+z^{-1}-2)t\right)^2-\frac{(z^{-1}X^3+z^{-1}-2)^2}{4}t^2+(1-z^{-1})t^2
\end{align*}
}

\noindent Multiply by 4 (char($\F_{q}\neq 2)$) and divide by $t^2$ to get that $\mathcal{C}_{z}$ is birationally equivalent to
$$\mathcal{C}: Y^2=F(X)$$
where 
$$Y=2G_{4}(1,X)t+G_{3}(1,X)$$
$$F(X)=G_{3}(1,X)^2-4G_{2}(1,X)G_{4}(1,X)$$

\noindent By substituting $G_{2}, G_{3}$ and $G_{4}$ in $F(X)$, and rescaling $Y\to z^{-1}Y$ we get the desired result, i.e.,  $\mathcal{C}_{z}$ is birationally equivalent to 
$$\mathcal{C}: Y^2=X^6+2(1-2z)X^3+1.$$
\end{proof}
\end{lemma}

In order to show that the L-polynomial of the curve $\mathcal{C}_{z}$ over $\fq$ is a perfect square, we will start by showing that the Jacobian of $\mathcal{C}_{z}$, Jac$(\mathcal{C}_{z})$, is isogenous to the product of two elliptic curves, i.e., that the Jac($\mathcal{C}_{z}$) is \emph{reducible}. To do that, it is convenient to find a slightly different model for our curve as we can see in the next criterion. First, we need to introduce the concept of equivalent curves. 

\begin{df}
 We say that two curves $Y^2=F(X)$ are equivalent if they are taken into one another by a fractional linear transformation of $X$ and the related transformation of $Y$ given by (\ref{frac_linear_transf}).
\end{df}

\noindent 
\begin{teo}[\cite{Cass96} Theorem 14.1.1]\label{criterion_genus_2}
The following properties of a curve $\mathcal{C}$ of genus 2 are equivalent:
\begin{enumerate}
 \item It is equivalent to a curve
\begin{equation}\label{criterion_condition1}
 Y^2=c_{3}X^6+c_2X^4+c_1X^2+c_0
\end{equation}
with no terms of odd degree in $X$.

\item It is equivalent to a curve
\begin{equation}
 Y^2=G_{1}(X)G_{2}(X)G_{3}(X)
\end{equation}
where the quadratics $G_{j}(X)$ are linearly dependent.

\item It is equivalent to
\begin{equation}
 Y^2=X(X-1)(X-a)(X-b)(X-ab)
\end{equation}
for some $a,b$.
\end{enumerate}

If one (and so all) of the previous conditions is satisfied, the Jacobian of $\mathcal{C}$ is reducible. 
\end{teo}

There are two maps of (\ref{criterion_condition1}) into elliptic curves
\begin{equation}
 \mathcal{E}_{1}: Y^2=c_3Z^3+c_2Z^2+c_1Z+c_0
\end{equation}
with $Z=X^2$ and 
\begin{equation}
 \mathcal{E}_{2}: V^2=c_0U^3+c_1U^2+c_2U+c_3
\end{equation}
with $U=X^{-2}, V=YX^{-3}$. These maps extend to maps of the Jacobian, which is therefore reducible (see \cite{Cass96}).
%\end{remark}

Hence, to apply Theorem \ref{criterion_genus_2} we find a different model for $\mathcal{C}_{z}$. In particular we will put our curve in form (\ref{criterion_condition1}).

\begin{lemma}\label{curve_no_odd_terms}
 The curve (\ref{curve_bir_equiv}) is equivalent to the curve
\begin{equation}\label{curve_equiv}
 Y^2= (1-z)X^6+3(2+z)X^4+3(3-z)X^2+z.
\end{equation}
\begin{proof}
 Consider the fractional linear transformation given by
$$X\to \frac{X+1}{X-1}$$
$$Y \to \frac{2Y}{(X-1)^3}$$
\end{proof}

\end{lemma}

\par Combining Lemma \ref{curve_no_odd_terms} and the observation at the end of Theorem \ref{criterion_genus_2}, we find two maps from (\ref{curve_equiv}) to the elliptic curves
\begin{equation}
 \mathcal{E}_{1,z}: Y^2: (1-z)Z^3+3(2+z)Z^2+3(3-z)Z+z
\end{equation}

\begin{equation}
 \mathcal{E}_{2,z}: V^2=zU^3+3(3-z)U^2+3(2+z)U+(1-z)
\end{equation}
Notice that $\mathcal{E}_{1,z}$ and $\mathcal{E}_{2,z}$ have discriminant $6912z(1-z)$, which is non-zero since $z\neq 0,1$.
Also, after rescaling, we can write

\begin{equation}\label{elliptic_curve_E1}
 \mathcal{E}_{1,z}: Y^2: Z^3+3(2+z)Z^2+3(3-z)(1-z)Z+z(1-z)^2
\end{equation}
and
\begin{equation}\label{elliptic_curve_E2}
 \mathcal{E}_{2,z}: V^2=U^3+3(3-z)U^2+3(2+z)zU+(1-z)z^2
\end{equation}

%%%%%%%%%%%%%%%%%%%%%%%%%%%%%%%%%%%%%%%%%%%%%%%%%%%%
%%%poner definition of twist of elliptic curve%%%%%%
%%%%%%%%%%%%%%%%%%%%%%%%%%%%%%%%%%%%%%%%%%%%%%%%%%%%
\vspace{0.3cm}

As we mentioned above, the existence of these two maps implies that Jac$(\mathcal{C}_{z})$ is isogenous to
$\mathcal{E}_{1,z}\times \mathcal{E}_{2,z}$. Next, we see that these elliptic curves are not totally independent of each other. In fact, one is isogenous to a twist of the other as we see in our next result.

\begin{prop}\label{elliptic_curves_isogenous}
 %For a prime $q\equiv 1 \pmod{3}$, 
The curve $\mathcal{E}_{1,z}$ is isogenous to the twisted curve $(\mathcal{E}_{2,z})_{-3}$. 
%over $\F_{q}$.

\begin{proof}
Consider the equation for the twisted curve $(\mathcal{E}_{2,z})_{-3}$:
\begin{equation}
(\mathcal{E}_{2,z})_{-3} : V^2=U^3-9(3-z)U^2+27(2+z)zU-27 (1-z)z^2
\end{equation}
Define $\varphi : \mathcal{E}_{1,z}\to (\mathcal{E}_{2,z})_{-3}$ such that $\varphi[0:1:0]=[0:1:0]$ and 
\begin{equation}\label{isogeny_genus2}
 \varphi[x:y:1] = \left[\frac{x^3+Ax^2+Bx+C}{(x+(z-1))^2}:\frac{(x^3+Dx^2+Ex+F)y}{(x+(z-1))^3}:1\right]
\end{equation}
where 
$\begin{cases} A=9 \\ B=3(1-z)(z+9)\\ C=(27-2z)(z-1)^2 \\ D=3(z-1)\\ E=3(z+15)(z-1)\\ F=(z-81)(z-1)^2.\end{cases}$\\

\noindent One can check by hand or with Maple for example, that the map $\varphi$ is well defined and gives an isogeny between the two curves.
\end{proof}
\end{prop}

\vspace{0.3cm}

Denote by $L(\mathcal{C}_{z}/\fq,T)$ the $L$-polynomial of $\mathcal{C}_{z}$ over $\fq$. Recall that we want to show that, for  $q \equiv 1 \pmod{3}$ we have $L(\mathcal{C}_{z}/\fq,T)=(1+aT+qT^2)^2$ for some $a \in \R$. So far, we have seen that 
$$L(\mathcal{C}_{z}/\fq,T)=(1+a_{1,q}(z)T+qT^2)(1+a_{2,q}(z)T+qT^2)$$
where $a_{1,q}(z)$ and $a_{2,q}(z)$ are the traces of Frobenius on the curves $\mathcal{E}_{1,z}$ and $\mathcal{E}_{2,z}$ respectively.
Therefore, we need to show that $a_{1,q}(z)=a_{2,q}(z)$, or equivalently, that $\# \mathcal{E}_{1,z}(\fq)=\# \mathcal{E}_{2,z}(\fq)$ for $q \equiv 1 \pmod{3}$. This is the statement of our next result.

\begin{cor}\label{same_points_ell_curves}
Let $q$ be a prime such that $q \equiv 1 \pmod{3}$. Then
$$\# \mathcal{E}_{1,z}(\F_{q})=\# \mathcal{E}_{2,z}(\F_{q}).$$

\begin{proof}
Fix $q$ in the conditions of the corollary. Let $a_{1,q}(z)$ and $a_{2,q}(z)$ be the traces of Frobenius on the elliptic curves $\mathcal{E}_{1,z}$ and $\mathcal{E}_{2,z}$ respectively, i.e.
$$\# \mathcal{E}_{1,z}(\F_{q})=q+1-a_{1,q}(z)$$
$$\# \mathcal{E}_{2,z}(\F_{q})=q+1-a_{2,q}(z)$$

\noindent Since $(\mathcal{E}_{2,z})_{-3}$ is a twist of $\mathcal{E}_{2,z}$ we have
$$\# (\mathcal{E}_{2,z})_{-3}(\F_{q})= 1+q-\left(\frac{-3}{q}\right)a_{2,q}(z)$$
where $\left(\frac{\cdot}{q}\right)$ is the Legendre symbol.

\noindent Now, by Proposition (\ref{elliptic_curves_isogenous}) we know that
$$\# (\mathcal{E}_{1,z})(\F_{q})=\# (\mathcal{E}_{2,z})_{-3}(\F_{q})$$ 
hence
$$a_{2,q}(z)=\left(\frac{-3}{q}\right)a_{1,q}(z).$$

\noindent To finish the proof, it only remains to see that $\left(\frac{-3}{q}\right)=1$ for all primes $q \equiv 1 \pmod{3}$.

\noindent Since the Legendre symbol is completely multiplicative on its top argument, we can decompose $\left(\frac{-3}{q}\right)=\left(\frac{-1}{q}\right)\left(\frac{3}{q}\right)$. Also
\begin{equation}\label{computation(-1/q)} \left(\frac{-1}{q}\right)=(-1)^{(q-1)/2}=\begin{cases} 1 & \textnormal{if} \,\,\, q\equiv 1 \pmod{4}\\ 
-1 &\textnormal{if}\,\,\, q\equiv 3 \pmod{4}. \end{cases}
\end{equation}
and
\begin{equation}\label{computation(3/q)} \left(\frac{3}{q}\right)=(-1)^{\lceil(q+1)/6\rceil}=\begin{cases} 1 & \textnormal{if} \,\,\, q\equiv 1,11 \pmod{12}\\ 
-1 &\textnormal{if}\,\,\, q\equiv 5,7 \pmod{12}. \end{cases}\end{equation}
We will divide the analysis in cases. First notice that since $q \equiv 1 \pmod{3}$ then $q$ must be congruent to either $1$ or $7$ $\pmod{12}$.  
\begin{itemize}
\item Suppose $q \equiv 1 \pmod{12}$ and therefore $\left(\frac{3}{q}\right)=1$ by (\ref{computation(3/q)}).
Also, since $q \equiv 1 \pmod{12}$, we have that $q \equiv 1 \pmod{4}$, hence $\left(\frac{-1}{q}\right)=1$ by (\ref{computation(-1/q)}).
Then $\left(\frac{-3}{q}\right)=1$ as desired.
\item Suppose $q \equiv 7 \pmod{12}$, then $\left(\frac{3}{q}\right)=-1$. Also, in this case $q \equiv 3 \pmod{4}$, and so $\left(\frac{-1}{q}\right)=-1$, giving that $\left(\frac{-3}{q}\right)=1$ as desired.
\end{itemize}
Hence 
$$\# \mathcal{E}_{1,z}(\F_{q})=\# \mathcal{E}_{2,z}(\F_{q}) \,\,\,\, \text{for all}\,\,\, q \equiv 1 \pmod{3}.$$

\end{proof}
\end{cor}
\vspace{0.2cm}

We have now all the necessary tools to complete the proof of Conjecture \ref{conject} for the case when $l=3$.

%%%%%%%%%%%%%%%%%%%%%%%%%%%%%%%%%%%%%%%%%%%%%%%%
%%%%%%%% proof conjecture for l=3%%%%%%%%%%%%%%%
%%%%%%%%%%%%%%%%%%%%%%%%%%%%%%%%%%%%%%%%%%%%%%%%

\begin{teo}\label{conjecture l=3}
Conjecture \ref{conject} is true for $l=3$.
\begin{proof}

First notice that when $l=3$ we have two different cases to consider, namely the curves with $(m,s)=(1,2)$ and $(m,s)=(2,1)$. However, by section \ref{hgf and ac} Corollary \ref{curves_have_same_number_points} these two curves have the same number of points over every finite field extension of $\F_{q}$, therefore they have the same zeta function over $\F_{q}$. Also, the hypergeometric functions that appear on the right hand side of equation (\ref{formula}) are the same for both curves. Because of these, it is enough to prove that the conjecture is true for one of these curves, say $\mathcal{C}_{z}: y^3=t(1-t)^2(1-zt)$.
As above, write the zeta function of $\mathcal{C}_{z}$ as
\begin{align*}
Z(\mathcal{C}_{z}/\F_{q};T)&=\frac{(1-\alpha_{1,q}(z)T)(1-\overline{\alpha_{1,q}(z)}T)(1-\alpha_{2,q}(z)T)(1-\overline{\alpha_{2,q}(z)}T)}{(1-T)(1-qT)}\\
& =\frac{(1-a_{1,q}(z)T+qT^2)(1-a_{2,q}(z)T+qT^2)}{(1-T)(1-qT)}
\end{align*}
where $a_{i,q}(z)=\alpha_{i,q}(z)+\overline{\alpha_{i,q}}(z)$.
Using the same notation as in equation (\ref{puntoscurva}), we have that
\begin{align}\label{formula1}
F_{1,q}(z)+ F_{2,q}(z) 
%&=-(\alpha_{1,q}(z)+\overline{\alpha_{1,q}}(z)+\alpha_{2,q}(z)+\overline{\alpha_{2,q}}(z))\nonumber\\
& =-(a_{1,q}(z)+a_{2,q}(z))
\end{align}

\noindent Recall that $F_{1,q}(z)=\,_2F_1[\eta_{q}^{2},\eta_{q};\varepsilon|z]$ and $F_{2,q}(z)=\, _{2}F_{1}[\eta_{q},\eta_{q}^{2};\varepsilon|z]$, therefore, Corollary \ref{conjugate_hpgf} in section \ref{preliminaries} implies that $F_{1,q}(z)=F_{2,q}(z)$. 
Also, as we have seen in Corollary \ref{same_points_ell_curves}, $a_{1,q}(z)=a_{2,q}(z)$.
Hence, (\ref{formula1}) becomes
$$2F_{1,q}(z)=-2a_{1,q}(z)$$ 
so $a_{1,q}(z)=-F_{1,q}(z)$ and $a_{2,q}(z)=-F_{2,q}(z)$, proving the conjecture for $l=3$. 
\end{proof}
\end{teo}

\vspace{0.3cm}

\section{Proof of Conjecture for $l=5$}\label{proof l=5}
\vspace{0.3cm}

Our next objective is to prove that Conjecture \ref{conject} also holds when $l=5$.
The proof has some ingredients in common with the previous case, however is not completely analogous and requires some different techniques as we will see. 

Consider the smooth projective curve with affine model
\begin{equation}\label{curve_l=5}
 \mathcal{C}_{z}: y^5=t(1-t)^4(1-zt)
\end{equation}
over a finite field $\F_{q}$ with $q$ prime, $q \equiv 1 \pmod{5}$ and $z\in\F_{q}\backslash\{0,1\}$.
Notice that, by performing the same transformations done in Lemma \ref{curve_biratl} and the fractional linear transformation
$$X\to \frac{X+1}{X-1}$$
$$Y \to \frac{2Y}{(X-1)^5}$$
on the curve (\ref{curve_l=5}) we get the following result.
\begin{lemma}
The curve $\mathcal{C}_{z}: y^5=t(1-t)^4(1-zt)$ is equivalent to the curve
\begin{equation}\label{curve_equiv_l=5}
 \mathcal{C}:Y^2= (1-z)X^{10}+(20+5z)X^8+(110-10z)X^6+(100+10z)X^4+(25-5z)X^2+z.
\end{equation}
\end{lemma}

\noindent Define the curves $\mathcal{H}_{1,z}: y^2=f(x)$ and $\mathcal{H}_{2,z}:y^2=g(x)$ where
\begin{equation}\label{H1}
f(x) = (1-z)x^{5}+(20+5z)x^4+(110-10z)x^3+(100+10z)x^2+(25-5z)x+z
\end{equation}
and
\begin{equation}\label{H2}
g(x)=zx^5+(25-5z)x^4+(100+10z)x^3+(110-10z)x^2+(20+5z)x+(1-z).
\end{equation}

\noindent Then, by the same argument in the previous section, we can find two maps from $\mathcal{C}$ to $\mathcal{H}_{1,z}$ and $\mathcal{H}_{2,z}$, and extending these maps to the Jacobians of the curves, we conclude that Jac($\mathcal{C}$) is isogenous to Jac($\mathcal{H}_{1,z}$)$\times$ Jac($\mathcal{H}_{2,z}$). We start by showing that the L-polynomial of $\mathcal{C}_{z}$ over $\F_{q}$ with $q \equiv 1 \pmod{5}$ is a perfect square. First, we recall some results about abelian varieties.

\vspace{0.3cm}
Let $k$ be a perfect field, which will eventually be finite. Recall that an \emph{abelian variety over $k$} is a subset of some projective $n$-space over $k$ which
\begin{enumerate}
 \item is defined by polynomial equations on the coordinates (with coefficients in $k$),
\item is connected, and
\item has a group law which is algebraic (i.e., the coordinates of the sum of two points are rational functions of the coordinates of the factors).
\end{enumerate}
We say that an abelian variety over $k$ is \emph{simple} if it has no nontrivial abelian subvarieties. We have the following result.

\begin{teo}(Poincar\'e-Weil) 
Every abelian variety over $k$ is isogenous to a product of powers of nonisogenous simple abelian varieties over $k$.
\end{teo}
 
\vspace{0.3cm}

Consider $\mathcal{C}_{z}: y^5=t(1-t)^4(1-zt)$ over the algebraically closed field $\overline{\Q}$ and let $\zeta:=e^{2\pi i/5}$ be a fifth root of unity. Then the map $[\zeta]:\mathcal{C}_{z}\to \mathcal{C}_{z}$ defined by $[\zeta](t,y)=(t,\zeta y)$ defines an automorphism on the curve $\mathcal{C}_{z}$. Denote $J_{z}:= \text{Jac}(\mathcal{C}_{z})$ and $J_{i,z}:=\text{Jac}(\mathcal{H}_{i,z})$, for $i=1,2$. 

The automorphism $[\zeta]$ induces a map from $J_{z}$ to itself, hence 
$$[\zeta]\in \text{End}(J_{z}).$$

\noindent On the other hand, as we mentioned above, we can find an isogeny over $\Q$
$$\phi: J_{1,z}\times J_{2,z}\to J_{z}.$$ Applying $\phi$ we get
$$\phi(J_{1,z})\subseteq J_{z}$$
where $J_{1,z}$ here denotes $J_{1,z}\times \{0\}$. Similarly
$$\phi(J_{2,z})\subseteq J_{z}.$$

We also have $$[\zeta](\phi(J_{i,z}))\subseteq J_{z}$$
for $i=1,2$.
\vspace{0.3cm}

Consider now the curve $y^5=t(1-t)^4(1-zt)$ defined over $\overline{\Q(z)}$.
We can apply to this curve the same argument we did before, and we can see that $J_{i,z}$ are simple abelian varieties over $\overline{\Q(z)}$. Otherwise, if $J_{i,z}$ is isogenous to the product of two elliptic curves, then, for all $z$ the L-polynomial would have two quadratic factors, which is not the case. (See example in section \ref{example} at the end of this section). Therefore, we have $\phi(J_{1,z})$ and $[\zeta](\phi(J_{1,z}))$ two simple abelian varieties inside $J_{z}$. By Poincar\'e complete reducibility theorem, we have that
either $\phi(J_{1,z})\cap [\zeta](\phi(J_{1,z}))$ is finite or $\phi(J_{1,z})=[\zeta](\phi(J_{1,z}))$.

\vspace{0.3cm}
\begin{itemize}
 \item Case 1: $\phi(J_{1,z})\cap [\zeta](\phi(J_{1,z}))$ is finite.\\
In this case, by dimension count we have 
$$[\zeta](\phi(J_{1,z}))+\phi(J_{1,z})=J_{z}.$$
Then, since $\phi(J_{1,z})$  and $\phi(J_{2,z})$ are simple abelian varieties, the Poincar\'e -Weil Theorem implies that  
$$[\zeta](\phi(J_{1,z})) \approx \phi(J_{2,z})$$
over $\Q(\zeta)$, where $\approx$ denotes isogeny. Notice that this isogeny will exist over any field containing a fifth root of unity, therefore, finite fields $\F_{q}$ with $q \equiv 1 \pmod{5}$ are fine. Then, we get that $[\zeta](\phi(J_{1,z}))$ is isogenous to $\phi(J_{2,z})$ over $\F_{q}$ for $q \equiv 1 \pmod{5}$.

\item Case 2: $[\zeta](\phi(J_{1,z}))=\phi(J_{1,z})$.\\
For this case, we recall first some facts about abelian varieties (for details see \cite{La83} or \cite{Mu85}). Suppose $A/\overline{k}$ is a simple abelian variety of dimension $g$, and denote $\Delta:= \text{End}_{\overline{k}}(A)\otimes _{\Z}\Q$. Then, Poincar\'e's complete reducibility Theorem implies that $\Delta$ is a division algebra. Also, from the theory of division algebras we know that the dimension of a division algebra over its center is a perfect square, hence, if $K=\{x \in \Delta : xa=ax\,\, \text{for all} \,\, a \in \Delta\}$ is the center of $\Delta$, we have $[\Delta : K]=d^2$ for some integer $d$. On the other hand, if $[K:\Q]=e$ then $de|2g$, moreover, in characteristic zero, we have that $d^2e|2g$. 

Now that we have reviewed the results we need we can go back to case 2. 
In this case, let $\Delta :=\text{End}(\phi(J_{1,z}))\otimes_{\Z}\Q$ and $K$ be its center. Applying the results above, we can assume that $[\Delta : K]=d^2$ and  $[K:\Q]=e$, for some integers $d$ and $e$. Since the dimension of $\phi(J_{1,z})=2$ and char$(\overline{\Q(z)})=0$, we have that $d^2e|4$. 

By assumption, we have that
$$[\zeta] \in \text{End}(\phi(J_{1,z})),$$
hence
$$\Q(\zeta) \subseteq \text{End}(\phi(J_{1,z}))\otimes _{\Z} \Q :=\Delta.$$
Now, we have $$\Q \subseteq \Q(\zeta)\subseteq \Delta,$$
$$[\Q(\zeta):\Q]=4$$
and
$$[\Delta: \Q]|4.$$
Therefore, $4=4[\Delta:\Q(\zeta)]$, hence $$\Q(\zeta)=\Delta$$
i.e., $\text{End}(\phi(J_{1,z}))\otimes _{\Z}\Q$ is a field of degree 4 over $\Q$.

Recall the following definition.
\begin{df}
 A totally imaginary quadratic extension of a totally real field is called a CM field (Complex Multiplication).
\end{df}

Then, $\Delta$ is a CM field, since it is equal to $\Q(\zeta)$ and every cyclotomic field is a CM field ($\Q \subseteq \Q(\zeta, \overline{\zeta}) \subseteq \Q(\zeta)$). 
%Also, since $\text{End}(\phi(J_{1,z}))\otimes _{\Z} \Q$ has maximum dimension 4, we say that the field has full CM.

%%%%%%%%%%%%%%%%%%%%%%%%%%%%%%%%%%%%%%%%%%%%%%%%%%%%%%%
%% definicion de constant family of ab, varieties%%%%%
%%%%%%%%%%%%%%%%%%%%%%%%%%%%%%%%%%%%%%%%%%%%%%%%%%%%%%%%

\begin{teo}[Shimura \cite{Sh05}]
Let $k$ be a field of characteristic zero. Over $k$ there do not exist non-constant families of abelian varieties with full CM (i.e., the endomorphism ring has maximal dimension). 
\end{teo}

However, our family $\phi(J_{1,z})$ is non-constant, as it can be computationally checked with Magma using Igusa invariants. 
\end{itemize}

\noindent Therefore, only case 1 is possible, and we have 
$$[\zeta](\phi(J_{1,z})) \approx \phi(J_{2,z}).$$

\vspace{0.4cm}
%%%%%%%%%%%%%%%%%%%%%%%%%%%%%%%%%%%%%%%%%%%%%%%%%%%%%%%%%%
%%%%%%%%%%%%%%%%%%conjecture for l=5%%%%%%%%%%%%%%%%%%%%%
%%%%%%%%%%%%%%%%%%%%%%%%%%%%%%%%%%%%%%%%%%%%%%%%%%%%%%%%%%

We now state and prove our theorem.
\begin{teo}\label{conjecture l=5}
%Let $p(z)$ and $\widehat{p}(z)$ be the polynomials defined above by equations (\ref{p}) and (\ref{widep}). If the polynomial $(p+\widehat{p})(z)$ satisfy the conditions in Lemma \ref{p(x)=p(1-x)} modulo $q$ for a prime $q \equiv 1 \pmod{5}$ and $\#\mathcal{H}_{1,z}(\F_{q^2})=\#\mathcal{H}_{2,z}(\F_{q^2})$, then 
Conjecture \ref{conject} holds for $l=5$ over $\F_{q}$, for a prime $q \equiv 1 \pmod{5}$.

\vspace{0.3cm}

\begin{proof}[Proof of Theorem \ref{conjecture l=5}]
 By the same argument done in the proof of Conjecture \ref{conject} for $l=3$, it is enough to prove the conjecture for the curve $\mathcal{C}_{z}:y^5=t(1-t)^4(1-zt)$.
Also, by the previous argument, then the L-polynomial of $\mathcal{C}_{z}$ is a perfect square, i.e., we can assume, after rearranging terms if necessary, that $a_{1,q}(z)=a_{4,q}(z)$ and $a_{2,q}(z)=a_{3,q}(z)$. We can write then 
$$Z(\mathcal{C}_{z}/\F_{q};T)= \frac{(1-a_{1}(z)T+qT^2)^2(1-a_{2}(z)T+qT^2)^2}{(1-T)(1-qT)}$$

\vspace{0.3cm}
\noindent By Corollary \ref{conjugate_hpgf} in section \ref{preliminaries}, we know that $F_{1,q}(z)=F_{4,q}(z)$ and $F_{2,q}(z)=F_{3,q}(z)$. 
%Also, since the zeta function of $\mathcal{C}_{z}$ is a perfect square, we can assume, after rearranging terms if necessary, that $a_{1}(z)=a_{4}(z)$ and $a_{2}(z)=a_{3}(z)$. 
At the end, we get that
\begin{equation}\label{a1+a2}
-(a_{1,q}(z)+a_{2,q}(z))=F_{1,q}(z)+F_{2,q}(z).
 \end{equation}
We want to prove that
$-a_{1,q}(z)=F_{1,q}(z)$ and $-a_{2,q}(z)=F_{2,q}(z)$. Recall, from (\ref{formulapuntos}) that
\begin{equation}\label{a1^2+a2^2}
-(a_{1,q}(z)^2-2q+a_{2,q}(z)^2-2q)=F_{1,q^2}(z)+F_{2,q^2}(z).
\end{equation}
Also, keep in mind that for the hypergeometric functions $F_{1,q}$ and $F_{2,q}$ we are choosing a character $\eta_{q} \in \fqmc$ of order $5$, and for the hypergeometric functions $F_{1,q^2}$ and $F_{2,q^2}$ the character we are choosing is in $\widehat{\F_{q^2}^{\times}}$, also of order $5$.

%%%%%%%%%%%%%%%%%%%%%%%%%%%%%%%%%%%%%%%%%%%%%%
%% equation between hpgf different fields%%%%%
%%%%%%%%%%%%%%%%%%%%%%%%%%%%%%%%%%%%%%%%%%%%%%

\begin{claim}
 It is enough to show that
\begin{equation}\label{relation_hypg_functions}
F_{i,q^2}(z)=-F_{i,q}(z)^2+2q
\end{equation}
for $ i=1,2$.

\begin{proof}[Proof of Claim]

We will write $a_{i,q}:=a_{i,q}(z)$ and $F_{i,q^k}:=F_{i,q^k}(z)$ for $i,k=1,2$ . 
If (\ref{relation_hypg_functions}) is true, from (\ref{a1+a2}) and (\ref{a1^2+a2^2}) we get the system of equations in $a_{1,q}$ and $a_{2,q}$
$$\begin{cases} -a_{1,q}-a_{2,q}= F_{1,q}+F_{2,q}\\ a_{1,q}^2+a_{2,q}^2=F_{1,q}^2+F_{2,q}^2.\end{cases}$$
which is equivalent to 
$$\begin{cases} -a_{1,q}-a_{2,q}= F_{1,q}+F_{2,q}\\ a_{1,q}a_{2,q}=F_{1,q}F_{2,q}.\end{cases}$$
hence, $a_{1,q}=-F_{1,q}$ and $a_{2,q}=-F_{2,q}$.
\end{proof}
\end{claim}
\noindent Continuing with the proof of the conjecture for $l=5$, it only remains to show that (\ref{relation_hypg_functions}) holds. For that, let's start by writing explicitly the functions we have on the left and right hand side of (\ref{relation_hypg_functions}). We start with $F_{1,q}=F_{4,q}=11 \,{}_{2}F_{1} \left(
\begin{array}{ll|}
\eta_{q}, & \eta_{q} ^{4} \\
     & \varepsilon \\
\end{array} \: z \right)$. 
The other case will be the result of a similar argument.
%%%%%%%%%%%%%%%%%%%%%%%%%%%%%%%%
%%%%%equation for F_{1,q}^2%%%%%
%%%%%%%%%%%%%%%%%%%%%%%%%%%%%%%%

%\allowdisplaybreaks{
\begin{align*}
%\label{equation_F1}
 F_{1,q}^2 &= \sum_{x,y \in \F_{q}} \eta_{q}^4(xy)\,\eta_{q}((1-x)(1-y))\,\eta_{q}^4((1-zx)(1-zy))\nonumber\\
& = \sum_{s \in \F_{q}^{\times}} \eta_{q}^4(s)\sum_{x\in\F_{q}^{\times}}\eta_{q}((1-x)(1-s/x))\,\eta_{q}^4((1-zx)(1-zs/x)) & (xy=s)\nonumber\\
& = \sum_{s \in \F_{q}^{\times}} \eta_{q}^4(s)\sum_{x\in\F_{q}^{\times}}\eta_{q}(1-x-s/x+s)\,\eta_{q}^4(1-z(x+s/x)+z^2s). 
\end{align*}
%}
\vspace{0.2cm}

\noindent On the other hand, define $\chi \in \widehat{\F_{q^2}^{\times}}$ such that $\chi:=\eta_{q}\circ N_{\F_{q}}^{\F_{q^2}}$, i.e., for $\alpha \in \F_{q^2}$, $\chi(\alpha)=\eta_{q}(N_{\F_{q}}^{\F_{q^2}}(\alpha))=\eta_{q}(\alpha^{q+1})$, where $N_{\F_{q}}^{\F_{q^2}}$ denotes the norm from $\F_{q^2}$ down to $\F_{q}$. Since $N_{\F_{q}}^{\F_{q^2}}(\alpha) \in \F_{q}$ for all $\alpha \in \F_{q^2}$ then $\chi$ is well defined and it actually defines a character of $\F_{q^2}^{\times}$ (see \cite{IR90} Chapter 11).. 
Moreover, since the order of $\eta_{q}$ is $5$ then the order of $\chi$ must divide $5$.
But if $x \in \fq$ then $N_{\F_{q}}^{\F_{q^2}}(x)=x^{q+1}=x^2$, therefore $\chi|_{\F_{q}}=\eta_{q}^2\neq \varepsilon$. Then $\chi 
\in \widehat{\F_{q^2}^{\times}}$ is a character of order $5$. We choose this character for our computations and we have

%%%%%%%%%%%%%%%%%%%%%%%%%%%%%%%%%
%%%%equation for F_{1,q^2}%%%%%%%
%%%%%%%%%%%%%%%%%%%%%%%%%%%%%%%%%

\begin{align*}
%\label{equation_G1}
 F_{1,q^2} &:= \sum_{c \in \F_{q^2}} \chi^4(c)\chi(1-c)\chi^4(1-zc)\nonumber\\
&= \sum_{c \in \F_{q^2}} \eta_{q}^4(c^{q+1})\eta_{q}((1-c)^{q+1})\eta_{q}^4((1-zc)^{q+1}) \nonumber \\
% & (\chi=\eta_{q}\circ N_{\F_{q}}^{\F_{q^2}})\\
& = \sum_{s \in \F_{q}^{\times}} \eta_{q}^4(s)
\sum_{\alpha \in\F_{q^2}^{\times}, \,\, \alpha^{q+1}=s}\eta_{q}(1-\alpha-s/\alpha +s)\eta_{q}^4(1-z(\alpha+s/\alpha)+z^2s)  %(c^{q+1}=s)
\end{align*}
where the last equality follows by putting $c^{q+1}=s$ and noting that, since char$(\F_{q})=q$  and $\alpha^{q+1}=s$ then
$$(1-\alpha)^{q+1}=(1-\alpha)^q(1-\alpha)=(1-\alpha ^q)(1-\alpha)=1-\alpha-\alpha^{q}+\alpha^{q+1}=1-\alpha-s/\alpha +s.$$
A similar computation gives that
$$(1-zc)^{q+1}=1-z(\alpha+s/\alpha)+z^2s.$$

\noindent For $s \in \F_{q}^{\times}$ define $h:\F_{q^2}^{\times}\to \F_{q^2}$ such that $h(t)=t+s/t$ and let $f$ and $g$ be the restrictions of $h$ to the sets $\F_{q}^{\times}$ and $N^{-1}(s):=\{\alpha \in \F_{q^2}: \alpha^{q+1}=s\}\subset \F_{q^2}^{\times}$ respectively, i.e.,
$$f:=h|_{\F_{q}^{\times}}:\F_{q}^{\times} \to \F_{q}$$
$$g:=h|_{N^{-1}(s)}:N^{-1}(s) \to \F_{q}$$
Notice that, if $\alpha \in N^{-1}(s)$ then $g(\alpha)=\alpha + s/\alpha=\alpha + \alpha^{q}=\text{tr}(\alpha) \in \F_{q}$, hence Im$(g)\subset \F_{q}$. Making use of these functions, we can rewrite
\begin{equation*}
 F_{1,q}^2 = \sum_{s \in \F_{q}^{\times}} \eta_{q}^4(s)\sum_{b \in \text{Im}(f)\subset \F_{q}} \eta_{q}(1-b+s)\eta_{q}^4(1-bz+z^2s)
\end{equation*}
and
\begin{equation*}
 F_{1,q^2}=\sum_{s \in \F_{q}^{\times}} \eta_{q}^4(s) \sum_{b \in \text{Im}(g)\subset \F_{q}} \eta_{q}(1-b+s)\eta_{q}^4(1-bz+z^2s) 
\end{equation*}

\noindent Combining both equations, we have
\begin{equation}\label{equation_with_b}
 F_{1,q}^2+F_{1,q^2}=\sum_{s \in \F_{q}^{\times}} \eta_{q}^{4}(s) \sum_{\text{some} \,\,b\in \F_{q}} \eta_{q}(1-b+s) 
\eta_{q}^{4}(1-bz+z^2s).
\end{equation}

\noindent Our next and last step will be to describe over what elements are we summing in the inner sum of (\ref{equation_with_b}). Fix $s \in \F_{q}^{\times}$. Note that $h$ is generically a 2-to-1 map. To see this, suppose $b \in \text{Im}(h)$, therefore there exists $t\in \F_{q^2}^{\times}$ such that $t+s/t=b$, or equivalently $t^2-bt+s=0$. Hence, $h$ is 2-to-1 except when $b^2-4s=0$, i.e., except when $s$ is a perfect square in $\F_{q}$.

\begin{itemize}
 \item  Case 1: $s$ is not a perfect square in $\F_{q}^{\times}$.\\
By previous comment, we know that in this case $h$ is 2-to-1 map. Also, is not too hard to show that $h$ is surjective when restricted to the two domains $\F_{q}^{\times}$ and $N^{-1}(s)$. Therefore, in this case every element $b \in \F_{q}^{\times}$ will appear exactly twice in the inner sum of (\ref{equation_with_b}).

\item Case 2: $s$ is a perfect square in $\F_{q}^{\times}$.\\
In this case, let $s=a^2$, then $b=2a$ or $b=-2a$. As in previous case, every $b \in \F_{q}$ different from $2a$ and $-2a$ will appear exactly twice in the inner sum of (\ref{equation_with_b}).
% since $\text{Im}(f)\cup \text{Im}(g)=\F_{q}$. 
What about $b=2a$ and $b=-2a$? If $s$ is a perfect square then $\text{Im}(f)\cap\text{Im}(g)=\{2a,-2a\}$, hence both $2a$ and $-2a$ will also appear twice in the sum, once as part of the sum for $F_{1,q}^2$ and once as part of the sum for $F_{1,q^2}$.
\end{itemize}

\noindent Summarizing we have
\begin{align*}
 F_{1,q}^2+F_{1,q^2} &= \sum_{\substack{s \in \F_{q}^{\times}\\ \left(\frac{s}{q}\right)=-1}} \eta_{q}^{4}(s) \sum_{\text{some}\,\,b\in \F_{q}} \eta_{q}(1-b+s) \eta_{q}^{4}(1-bz+z^2s)\\
& \qquad \qquad + \sum_{\substack{s \in \F_{q}^{\times}\\  \left(\frac{s}{q}\right)=1}} \eta_{q}^{4}(s) \sum_{\text{some}\,\,b\in \F_{q}} \eta_{q}(1-b+s) \eta_{q}^{4}(1-bz+z^2s)\\
& = 2\sum_{\substack{s \in \F_{q}^{\times}\\ \left(\frac{s}{q}\right)=-1}} \eta_{q}^{4}(s) \sum_{b\in \F_{q}} \eta_{q}(1-b+s) \eta_{q}^{4}(1-bz+z^2s).\\
& \qquad \qquad + 2\sum_{\substack{s \in \F_{q}^{\times}\\  \left(\frac{s}{q}\right)=1}} \eta_{q}^{4}(s) \sum_{b\in \F_{q}} \eta_{q}(1-b+s) \eta_{q}^{4}(1-bz+z^2s)\\
&= 2\sum_{s \in \F_{q}^{\times}} \eta_{q}^{4}(s) \sum_{b\in \F_{q}} \eta_{q}(1-b+s) \eta_{q}^{4}(1-bz+z^2s).
\end{align*}

\noindent To finish the proof we need to see that 
$$\sum_{s \in \F_{q}^{\times}} \eta_{q}^{4}(s) \sum_{b\in \F_{q}} \eta_{q}(1-b+s) \eta_{q}^{4}(1-bz+z^2s)=q.$$

\noindent We begin by rewriting the inner sum in the above formula, but first recall that the action of $GL_{2}(\F_{q})$ on $\F_{q}$ given by 
\begin{displaymath}
 \left( \begin{array}{cc}
a & b\\
c & d
\end{array}
\right)\cdot w:=\frac{a w +b}{c w+d}
\end{displaymath}
defines an automorphism of $\mathbb{P}^{1}(\F_{q})$. Now, since $\eta_{q}^5=\varepsilon$ and $\eta_{q}(0)=0$ we get
\begin{align*}
\sum_{b\in \F_{q}} \eta_{q}(1-b+s) \eta_{q}^{4}(1-bz+z^2s) &= \sum_{\substack{b \in \F_{q}\\ b\neq (z^{-1}+zs)}} 
         \eta_{q}\left(\frac{1-b+s}{1-bz+z^2s}\right)\\
&= \sum_{\substack{b \in \F_{q}\\ b\neq (z^{-1}+zs)}} 
         \eta_{q}(\gamma \cdot b)
\end{align*}
where $\gamma:=\left( \begin{array}{cc}
-1 & s+1\\
-z & z^2s+1
\end{array}
\right)$. Now, $\text{det}\gamma=(z-1)(1-sz)$, therefore, since $z\neq 1$ we see that as long as $s\neq z^{-1}$, $\gamma$ defines an automorphism of $\mathbb{P}^{1}(\F_{q})$. Then, by separating the sums according to whether $s=z^{-1}$ or not, we have:

\begin{align*}
 \sum_{s \in \F_{q}^{\times}} \eta_{q}^{4}(s) \sum_{\substack{b\in \F_{q}\\ b\neq (z^{-1}+zs)}} \eta_{q}(\gamma\cdot b) &=
  \sum_{\substack{s\in\F_{q}^{\times}\\ s \neq z^{-1}}}  \eta_{q}^{4}(s) \sum_{\substack{b\in \F_{q}\\ b\neq (z^{-1}+zs)}} \eta_{q}(\gamma \cdot b) \\ &\qquad \qquad + \eta^{4}_{q}(z^{-1})\sum_{\substack{b\in \F_{q}\\ b\neq (z^{-1}+1)}} \eta_{q}\left(\frac{1-b+z^{-1}}{1-bz+z}\right) \\
&= A+B
\end{align*}
where $A$ and $B$ are set to be the two sums appearing in the previous line. We now compute $A$ and $B$. First we have
\begin{align*}
 B &= \sum_{\substack{b\in \F_{q}\\ b\neq (z^{-1}+1)}} \eta^{4}_{q}(z^{-1}) \eta_{q}\left(\frac{1-b+z^{-1}}{1-bz+z}\right) \\
& = \sum_{\substack{b\in \F_{q}\\ b\neq (z^{-1}+1)}} \eta_{q}\left(\frac{z-bz+1}{1-bz+z}\right) & (\eta_{q}^{4}(z^{-1})=\eta_{q}(z))\\
& = \sum_{\substack{b\in \F_{q}\\ b\neq (z^{-1}+1)}} 1 \\
& = q-1.
\end{align*}

\noindent Now we compute $A$. Since in this case the action of $\gamma$ defines an automorphism of $\mathbb{P}^{1}(\F_{q})$, and since $\gamma \cdot b$ runs over $\F_{q}- \{z^{-1}\}$ as $b$ runs over $\F_{q}-\{z^{-1}+sz\}$ we see that
%\allowdisplaybreaks{
\begin{align*}
 A & =  \sum_{\substack{s \in \F_{q}^{\times}\\ s\neq z^{-1}}} \eta_{q}^{4}(s) \sum_{\substack{u\in \F_{q}\\ u\neq z^{-1}}} \eta_{q}(u) \\
& = (-\eta_{q}^{4}(z^{-1})) (-\eta_{q}(z^{-1})) & \text{(orthogonality relations for characters)} \\
& = 1 & (\eta_{q}^{5}=\varepsilon)
\end{align*}
%}
Therefore, combining our calculations for $A$ and $B$ we see that

\begin{equation}
 F_{1,q}^2+F_{1,q^2} = 2(A+B)= 2q
\end{equation}
finishing the proof.
\end{proof}
\end{teo}

\begin{example}\label{example}
We illustrate with an example the result of the conjecture. Consider the smooth projective curve with affine model given by
$$\mathcal{C}_{3}: y^5=t^2(1-t)^3(1-3t)^2$$
over the finite field $\F_{11}$. $\mathcal{C}_{3}$ is a hyperelliptic curve of genus $4$, and using Magma we can compute its zeta function. We have that
$$Z(\mathcal{C}_{3}/\F_{11},T)= \frac{(121T^4+66T^3+26T^2+6T+1)^2}{(1-T)(1-11T)}.$$

\noindent Therefore, after doing some algebra, we find the values of $a_{i,11}(3)$ for $i=1,\dots,4$.
Specifically, if $\zeta_{5}:=e^{2\pi i/5}$ we have
\begin{align*}
 a_{1,11}(3)= a_{4,11}(3)&=-4-2\zeta_{5}^2-2\zeta_{5}^3\\
 a_{2,11}(3)=a_{3,11}(3)&=-2+2\zeta_{5}^2+2\zeta_{5}^3.
\end{align*}

\vspace{0.3cm}

\noindent On the other hand, consider the multiplicative character $\eta_{11} \in \widehat{\F_{11}^{\times}}$ defined by $\eta_{11}(a):=\zeta_{5}$, where $a$ is a primitive element of $\F_{11}^{\times}$, i.e., $a$ generates $\F_{11}^{\times}$, and recall that 
$F_{i,11}(3)= 11\,{}_{2}F_{1}[\eta_{11} ^{3i},\eta_{11} ^{2i};\varepsilon|3]$.
Using Magma we get
\begin{align*}
 F_{1,11}(3)=F_{4,11}(3)&=4+2\zeta_{5}^2+2\zeta_{5}^3\\
F_{2,11}(3)=F_{3,11}(3)&=2-2\zeta_{5}^2-2\zeta_{5}^3.
\end{align*}

\noindent Hence $$F_{i,11}(3)=-a_{i,11}(3)\,\,\,\,\, \text{for all}\,\, i=1,2,3,4.$$

\end{example}
\vspace{0.3cm}

\section{Advances toward the general case}\label{generalconj}
\vspace{0.4cm}

Even though it is still work in progress to prove the conjecture in its full generality, some advances have already been made toward it. To show these advances is the purpose of this section. 

\vspace{0.3cm}

Suppose now that $l$ and $q$ are odd primes, with $q \equiv 1 \pmod{l}$, and let $z \in \F_{q}\backslash\{0,1\}$. Recall that our conjecture relates values of certain hypergeometric functions over $\F_{q}$ to counting points on certain curves over $\F_{q}$. Recall also, that the curves we are interested in are smooth projective curves of genus $l-1$ with affine model
$$\mathcal{C}_{z}^{(m,s)}: y^l=t^m(1-t)^s(1-zt)^m$$
where $1\leq m,s<l$ are integers such that $m+s=l$.
Now, as we mentioned in the previous section, Corollary \ref{curves_have_same_number_points} in section \ref{hgf and ac} states that the curves $\mathcal{C}_{z}^{(m,s)}$ have all the same number of points over every finite extension of $\F_{q}$ as $(m,s)$ varies over all pairs of positive integers with $m+s=l$, hence they all have the same zeta function over $\F_{q}$. This, together with the fact that the hypergeometric functions that appear on the right hand side of equation (\ref{formula}) are the same for all these curves imply that it is enough to prove the conjecture for only one of them, say
\begin{equation}\label{generic_curve}
\mathcal{C}_{z}^{1,l-1}: y^l=t(1-t)^{l-1}(1-zt).
\end{equation}
Throughout this section, we will denote this curve by $\mathcal{C}_{z}$.

\vspace{0.3cm}

So, the question is: what results would be enough to know in order to prove the conjecture for all primes $l$ and $q$ with $q \equiv 1 \pmod{l}$?
Recall that, by equations (\ref{formulapuntos}) and (\ref{puntoscurva}) we have that 
\begin{equation}\label{F_and_alpha}
 F_{1,q^n}(z)+F_{2,q^n}(z)+ \cdots +F_{l-1,q^n}(z) = -\sum_{i=1}^{l-1}(\alpha_{i,q}^n(z)+\overline{\alpha_{i,q}^n}(z))
\end{equation}
where $F_{i,q^n}(z)= q^n \,{}_{2}F_{1} \left(
\begin{array}{ll|}
\eta_{q^n} ^{i}, & \eta_{q^n} ^{i(l-1)} \\
     & \varepsilon \\
\end{array} \: z \right)$ with $\eta_{q^n} \in \widehat{\F_{q^n}^{\times}}$ a character of order $l$, 
and $\alpha_{i,q}(z)$ are the reciprocals of the roots of the zeta function of $\mathcal{C}_{z}$ over $\F_{q}$, i.e.,
$$Z(\mathcal{C}_{z}/\fq;T)=\frac{(1-\alpha_{1,q}(z)T)(1-\overline{\alpha_{1,q}(z)}T)\cdots(1-\alpha_{l-1,q}(z)T)
(1-\overline{\alpha_{l-1,q}(z)}T)}{(1-T)(1-qT)}.$$

From now on we will omit the dependency on $z$ of the hypergeometric functions and the roots of the zeta function, therefore, we will denote
$F_{i,q^n}:=F_{i,q^n}(z)$ and $\alpha_{i,q}:=\alpha_{i,q}(z)$. Also, as in the previous section, denote $a_{i,q}:=\alpha_{i,q}+\overline{\alpha_{i,q}}$,  for $i=1,\cdots, l-1$. Since we want to relate the hypergeometric functions above with the values $a_{i,q}$, first we are going to express the values $\alpha_{i,q}^n+\overline{\alpha_{i,q}^n}$ in terms of $a_{i,q}$ and $q$. We have the following theorem:

\begin{lemma}\label{Dickson}
For $\alpha \in \C$ such that $|\alpha|=\sqrt{q}$ denote $\alpha+\overline{\alpha}:=a$, and let $n$ be a non-negative integer. Then:
\begin{equation}
 \alpha^n+\overline{\alpha}^n=\sum_{i=0}^{\lfloor \frac{n}{2}\rfloor} (-1)^i\, T(n,i)\,q^i\, a^{n-2i}
\end{equation}
where  $T(0,0):=2$, $T(n,0):=1$ for $n>0$ and 
$$T(n,i):= \frac{n(n-i-1)!}{i!(n-2i)!},\,\,\,\, \text{for}\,\, n>0, \, i\geq 0.$$

\begin{proof}
 We will prove the result by induction on $n$.
For $n=0$ and $n=1$ is clear.

\noindent Now suppose the result is true for all $k\leq n$. We want to show then that is also true for $n+1$.
Notice that 
\begin{align*}
 (\alpha^n+\overline{\alpha}^n)(\alpha+\overline{\alpha}) &= \alpha^{n+1}+\overline{\alpha}^{n+1}+\alpha^n \overline{\alpha}+\overline{\alpha}^n \alpha\\
& = \alpha^{n+1}+\overline{\alpha}^{n+1}+\alpha \overline{\alpha} (\alpha^{n-1}+\overline{\alpha}^{n-1})\\
& =  \alpha^{n+1}+\overline{\alpha}^{n+1}+q (\alpha^{n-1}+\overline{\alpha}^{n-1}) & (\alpha \overline{\alpha}=q)
\end{align*}
Hence, since $a=\alpha+\overline{\alpha}$, we have  
\begin{equation}\label{powers_alpha}
  \alpha^{n+1}+\overline{\alpha}^{n+1}= (\alpha^n+\overline{\alpha}^n)\,a-q\,(\alpha^{n-1}+\overline{\alpha}^{n-1}).
\end{equation}

\noindent Combining equation (\ref{powers_alpha}) and the inductive hypothesis we have
\begin{align}\label{first_sum}
 \alpha^{n+1}+\overline{\alpha}^{n+1} &= \sum_{i=0}^{\lfloor\frac{n}{2}\rfloor}(-1)^i\, T(n,i)\,q^i\, a^{n+1-2i}-
                          \sum_{i=0}^{\lfloor\frac{n-1}{2}\rfloor}(-1)^i\, T(n-1,i)\,q^{i+1}\, a^{n-1-2i}\nonumber \\
& = a^{n+1}+\sum_{i=1}^{\lfloor\frac{n}{2}\rfloor}(-1)^i\, T(n,i)\,q^i\, a^{n+1-2i} \nonumber\\ 
& \qquad \qquad -\sum_{j=1}^{\lfloor\frac{n-1}{2}\rfloor+1}(-1)^{j-1}\, T(n-1,j-1)\,q^{j}\, a^{n+1-2j}
\end{align}
after breaking apart the $i=0$ contribution in the first sum, and making the change of variables $i+1=j$ in the second sum.

\noindent Now we separate in two cases.
\begin{itemize}
 \item  Case 1: $n$ is even.\\
Notice that, in this case we have that $\lfloor \frac{n}{2}\rfloor=\lfloor \frac{n-1}{2}\rfloor+1$. Then, equation (\ref{first_sum}) becomes
\allowdisplaybreaks{
\begin{align*}
 \alpha^{n+1}+\overline{\alpha}^{n+1}&= a^{n+1}+\sum_{i=1}^{\lfloor\frac{n}{2}\rfloor} (-1)^i\, 
\left(\frac{n}{i!(n-2i)!}+\frac{(n-1)}{(i-1)!(n+1-2i)!}\right)\\
& \qquad \qquad \qquad \cdot (n-1-i)!\,q^i\,a^{n+1-2i}\\
&= a^{n+1}+\sum_{i=1}^{\lfloor\frac{n}{2}\rfloor} (-1)^i\, T(n+1,i)\,q^i\,a^{n+1-2i}\\
&= \sum_{i=0}^{\lfloor\frac{n+1}{2}\rfloor} (-1)^i\, T(n+1,i)\,q^i\,a^{n+1-2i}
\end{align*} 
}after replacing $T(n,k)$ by its definition, doing some algebra, and noticing that, if $n$ is even then $\lfloor\frac{n}{2}\rfloor=\lfloor\frac{n+1}{2}\rfloor$. This proves the lemma for $n$ even.

\item Case 2: $n$ is odd.\\
In this case we have $\lfloor\frac{n}{2}\rfloor=\lfloor\frac{n-1}{2}\rfloor$ and $\lfloor\frac{n+1}{2}\rfloor=\lfloor\frac{n}{2}\rfloor+1$. Combining these, breaking apart the contribution of $i=\lfloor\frac{n}{2}\rfloor+1$ in the second sum, and using the previous computation, equation (\ref{first_sum}) becomes
\begin{align*}
 \alpha^{n+1}+\overline{\alpha}^{n+1}&= a^{n+1}+ \sum_{i=1}^{\lfloor\frac{n+1}{2}\rfloor-1} (-1)^i\, T(n+1,i)\,q^i\,a^{n+1-2i}\\
& \qquad  +(-1)^{\lfloor\frac{n+1}{2}\rfloor} \frac{(n-1)(n-1-\lfloor\frac{n+1}{2}\rfloor)!}{(\lfloor\frac{n+1}{2}\rfloor)-1)!(n+1-2\lfloor\frac{n+1}{2}\rfloor)!}\,
q^{\lfloor\frac{n+1}{2}\rfloor}\,a^{n+1-2\lfloor\frac{n+1}{2}\rfloor}.
\end{align*}
To finish the proof, we need to see that 
\begin{equation}\label{dickson_coeff}
\frac{(n-1)(n-1-\lfloor\frac{n+1}{2}\rfloor)!}{(\lfloor\frac{n+1}{2}\rfloor)-1)!(n+1-2\lfloor\frac{n+1}{2}\rfloor)!}
= T(n+1,\lfloor\frac{n+1}{2}\rfloor).
\end{equation}
This is not a hard computation. Write $n=2m+1$ for some $m \in \N$, then $\lfloor\frac{n+1}{2}\rfloor=m+1$. Substituting this in equation (\ref{dickson_coeff}) we get $2=2$ finishing the proof for $n$ odd.
\end{itemize}
\end{proof}
\end{lemma}

Now, equation (\ref{F_and_alpha}) and Lemma \ref{Dickson} allow us to relate explicitly the hypergeometric functions with the traces of Frobenius, giving
\begin{equation}\label{relation_hpgf_traces}
 F_{1,q^n}+F_{2,q^n}+ \cdots +F_{l-1,q^n} =-\sum_{i=1}^{l-1} \sum_{j=0}^{\lfloor \frac{n}{2}\rfloor} (-1)^j\, T(n,j)\,q^j\, a_{i,q}^{n-2j}.
\end{equation}
Since in Conjecture \ref{conject} we want to prove that $F_{i,q}=-a_{i,q}$ for all $i=1,\ldots,l-1$, then we have the following result.

%%%%%%%%%%%%%%%%%%%%%%%%%%%%%%%%%%%%%%%%%%%%%%%%
%% theorem: what is enough to prove conjecture%%
%%%%%%%%%%%%%%%%%%%%%%%%%%%%%%%%%%%%%%%%%%%%%%%%

\begin{prop}\label{relation_hypergeometric_functions}
If for all $i,n=1,\ldots,l-1$ we have that
\begin{equation}\label{relation_powers_hpgf}
 F_{i,q^n}=(-1)^{n+1}\sum_{j=0}^{\lfloor \frac{n}{2}\rfloor} (-1)^j\, T(n,j)\,q^j\, F_{i,q}^{n-2j}
\end{equation}
then Conjecture \ref{conject} is true.
\begin{proof}
 Assume equation (\ref{relation_powers_hpgf}) is true. Then, substituting into equation (\ref{relation_hpgf_traces}) for $n=1,\ldots,l-1$ we get a system of equations relating sums of the hypergeometric functions $F_{i,q}$ and their powers to sums of the traces of Frobenius $a_{i,q}$ and their powers. After simplifying this system of equations, we get an equivalent one of the form

\begin{equation}\label{system1}
 \begin{cases}
F_{1,q}+F_{2,q}+\cdots+F_{l-1,q}=-(a_{1,q}+a_{2,q}+\cdots+a_{l-1,q})\\
F_{1,q}^2+F_{2,q}^2+\cdots+F_{l-1,q}^2=a_{1,q}^2+a_{2,q}^2+\cdots+a_{l-1,q}^2\\
\hspace{4.5cm} \vdots \\
F_{1,q}^{n}+F_{2,q}^{n}+\cdots+F_{l-1,q}^{n}=(-1)^{n}(a_{1,q}^{n}+a_{2,q}^{n}+\cdots+a_{l-1,q}^{n})\\
\hspace{4.5cm} \vdots \\
F_{1,q}^{l-1}+F_{2,q}^{l-1}+\cdots+F_{l-1,q}^{l-1}=a_{1,q}^{l-1}+a_{2,q}^{l-1}+\cdots+a_{l-1,q}^{l-1}.
\end{cases}
\end{equation}

\vspace{0.4cm}

\noindent This fact can be seen by induction. For $n=1$ there is nothing to prove. Suppose now that 
$F_{1,q}^{k}+F_{2,q}^{k}+\cdots+F_{l-1,q}^{k}=(-1)^{k}(a_{1,q}^{k}+a_{2,q}^{k}+\ldots+a_{l-1,q}^{k})$ for all $k< n$. Now, by equation (\ref{relation_powers_hpgf}) we have
\begin{align}\label{eq1}
F_{1,q^n}+\cdots+F_{l-1,q^n}&= (-1)^{n+1}\sum_{j=0}^{\lfloor \frac{n}{2}\rfloor} (-1)^j\, T(n,j)\,q^j\, 
\left(F_{1,q}^{n-2j}+\cdots+F_{l-1,q}^{n-2j}\right)\nonumber\\
& =(-1)^{n+1}(F_{1,q}^{n}+\cdots+F_{l-1,q}^{n})\nonumber\\
& \qquad + (-1)^{n+1}\sum_{j=1}^{\lfloor \frac{n}{2}\rfloor} (-1)^j\, T(n,j)\,q^j\, 
\left(F_{1,q}^{n-2j}+\cdots+F_{l-1,q}^{n-2j}\right)\nonumber\\
& =(-1)^{n+1}(F_{1,q}^{n}+\cdots+F_{l-1,q}^{n})\nonumber\\
& \qquad +(-1)^{n+1}\sum_{j=1}^{\lfloor \frac{n}{2}\rfloor} (-1)^{n-j}\, T(n,j)\,q^j\, 
\left(a_{1,q}^{n-2j}+\cdots+a_{l-1,q}^{n-2j}\right)
\end{align}
where the last equality follows from the inductive hypothesis.

\noindent On the other hand, by breaking apart the contribution of $j=0$ in equation (\ref{relation_hpgf_traces}) we have
\begin{align}\label{eq2}
F_{1,q^n}+ \cdots +F_{l-1,q^n} &=-(a_{1,q}^n+\ldots+a_{l-1,q}^n)
+\sum_{j=1}^{\lfloor \frac{n}{2}\rfloor} (-1)^{j+1}\, T(n,j)\,q^j\, (a_{1,q}^{n-2j}+\ldots+a_{l-1,q}^{n-2j}).
\end{align}
Therefore, combining equations (\ref{eq1}) and (\ref{eq2}) and noticing that $(-1)^{2n+1-j}=(-1)^{j+1}$ we get
$$F_{1,q}^n+ \cdots +F_{l-1,q}^n =(-1)^n \left( a_{1,q}^n+\ldots+a_{l-1,q}^n\right)$$
as desired.

\noindent Next, by using the Newton-Girard formulas, which give relations between elementary symmetric polynomials and power sums, we see that the system (\ref{system1}) is equivalent to
\allowdisplaybreaks{
\begin{equation*}\label{system2}
 \begin{cases}
F_{1,q}+\cdots+F_{l-1,q}=-(a_{1,q}+a_{2,q}+\cdots+a_{l-1,q})\\
\sum_{1\leq i<j\leq l-1}F_{i,q}F_{j,q}=\sum_{1\leq i<j\leq l-1}a_{i,q}a_{j,q}\\
\hspace{3.5cm} \vdots \\
F_{1,q}\ldots F_{l-1,q}=a_{1,q}\ldots a_{l-1,q}
\end{cases}
\end{equation*}
}
\vspace{0.2cm}

\noindent i.e., the elementary symmetric polynomials in the variables $F_{1,q},\ldots, F_{l-1,q}$ equal (up to a sign) the elementary symmetric polynomials in $a_{1,q},\ldots,a_{l-1,q}$. Then, we can think of these values as being roots of the same polynomial, therefore, after rearranging terms, we have that
$$F_{i,q}=-a_{i,q}\,\,\,\,\, \text{for all}\,\, i=1,\ldots,l-1$$
and Conjecture \ref{conject} follows.
\end{proof}
\end{prop}

\begin{remark}
 Notice that it is enough to prove equation (\ref{relation_powers_hpgf}) only for prime powers of $q$, i.e., only for $1\leq n\leq l-1$ with $n$ prime. Otherwise, if $n=mr$ then 
$\F_{q^{n}}| \F_{q^{m}}| \F_{q}$ is a tower of extensions, and we can use the relation for these extensions of lower degree.
\end{remark}

\vspace{0.3cm}

As we have seen above, proving equation (\ref{relation_powers_hpgf}) for prime powers of $q$ would be enough to prove Conjecture \ref{conject} in the general case. However, equation (\ref{relation_powers_hpgf}) gets complicated as $n$ grows, so it would be helpful if this equation is needed for even fewer values of $n$ in order to prove the conjecture. This might be possible; in fact, this is what we did to prove cases $l=3$ and $l=5$ in section \ref{conjecture}. Hence, looking at the proofs in previous section, we see that 

\begin{prop}\label{relation_half_terms}
If the $L$-polynomial of the smooth projective curve of genus $l-1$ with affine model $\mathcal{C}_{z}:y^l=t(1-t)^{l-1}(1-zt)$ is a perfect square, and equation (\ref{relation_powers_hpgf}) is verified for all primes $n$ such that $1\leq n\leq (l-1)/2$, then Conjecture \ref{conject} holds.. 
\begin{proof}
Recall that 
$$L(\mathcal{C}_{z}/\F_{q};T)=\prod_{i=1}^{l-1} (1-a_{i,q}T+qT^2).$$
Hence, if the proposition is true, we would have
$$L(\mathcal{C}_{z}/\F_{q};T)=\prod_{i=1}^{(l-1)/2} (1-a_{i,q}T+qT^2)^2$$
therefore, after rearranging terms, we have $a_{i,q}=a_{l-i,q}$, for all $i=1,\cdots,l-1$. 

\noindent On the other hand, recall that by Corollary \ref{conjugate_hpgf} in section \ref{preliminaries} the hypergeometric functions $F_{i,q}$ come in pairs, i.e., $F_{i,q}=F_{l-i,q}$ for $i=1,\cdots, l-1$. Hence, system (\ref{system1}) gets reduced to half of it, having only $(l-1)/2$ unknowns. Then, it is enough to prove relation (\ref{relation_powers_hpgf}) only for primes $n$ up to $(l-1)/2$ in order to prove Conjecture \ref{conject}.
\end{proof}
\end{prop}

Now, the question is how can we determine if the $L$-polynomial of $\mathcal{C}_{z}$ over $\F_{q}$ is a perfect square. One possible way is to do an argument similar to the one done for the cases $l=3$ and $l=5$. First, notice that we have the following result, analogous to Lemma \ref{curve_biratl}.

\begin{teo}\label{birational_curve_general_case}
 The curve $\mathcal{C}_{z}: y^l=t(1-t)^{l-1}(1-zt)$ is birationally equivalent to
\begin{equation}\label{curve_bir_equiv_general}
\mathcal{C}: Y^2=X^{2l}+2(1-2z)X^l+1.
\end{equation}
\begin{proof}
 The proof is analogous to the proof of Lemma \ref{curve_biratl}.
\end{proof}
\end{teo}

Also, analogous to Lemma \ref{curve_no_odd_terms}, by considering the fractional linear transformation
$$X\to \frac{X+1}{X-1}$$
$$Y \to \frac{Y}{(X-1)^l}$$
we see that the curve (\ref{curve_bir_equiv_general}) is equivalent to a curve of the form 
$$Y^2= c_{l}X^{2l}+c_{l-1}X^{2(l-1)}+\cdots+c_{1}X^2+c_0$$
with no terms of odd degree in X, where the coefficients $c_{i}$ are polynomial equations in $z$. Then, as in previous section, we can conclude that the jacobian of $\mathcal{C}_{z}$ is isogenous to the product of the jacobians of two curves of genus $(l-1)/2$, call them $\mathcal{H}_{1,z}$ and $\mathcal{H}_{2,z}$. 
Therefore, by Proposition \ref{relation_half_terms}, we have

\begin{teo}
Let $q \equiv 1 \pmod{l}$. If $\#\mathcal{H}_{1,z}(\F_{q^i})=\#\mathcal{H}_{2,z}(\F_{q^i})$ for all $i=1,\ldots, (l-1)/2$, and equation (\ref{relation_powers_hpgf}) holds for all primes $n$ such that $1\leq n\leq (l-1)/2$, then Conjecture \ref{conject} holds.

\begin{proof}
 Notice that the fact that $\#\mathcal{H}_{1,z}(\F_{q^i})=\#\mathcal{H}_{2,z}(\F_{q^i})$ for all $i=1,\ldots, (l-1)/2$ implies that the curves $\mathcal{H}_{1,z}$ and $\mathcal{H}_{1,z}$ have the same L-polynomial over $\F_{q}$, then, as we mentioned above, the system (\ref{system1}) gets reduced to half of it, having only $(l-1)/2$ unknowns. The rest of the proof follows from Proposition \ref{relation_half_terms}.
\end{proof}

\begin{remark}
 Notice that, if Proposition \ref{relation_hypergeometric_functions} holds (i.e. Conjecture \ref{conject} is true over $\F_{q}$), using Lemma \ref{Dickson} we can get a result similar to Conjecture \ref{conject} over $\F_{q^n}$, for $n\in \N$.
\end{remark}

\end{teo}

\vspace{0.4cm}

\section{acknowledgments}
%{\it aknowledgment:}
\vspace{0.3cm}
The author would like to thank her advisor Matthew Papanikolas for suggesting this problem and for his advise and support during the preparation of this paper.
\vspace{0.4cm}
%{\it aknoledgment:}


\begin{thebibliography}{99}
%\bibitem{Ah02}
% S. Ahlgren,
%  \emph{The points of a certain fivefold over finite fields and the twelfth power of the eta function},
%  Finite Fields Appl. \textbf{8} (2002),
%  18-33.
\bibitem{AO00}
  S. Ahlgren and K. Ono,
  ``Modularity of a certain Calabi-Yau threefold,''
  \emph{Monatsh. Math.}, vol. 129,
  pp. 177-190, 2000.

\bibitem{As75}
  R. Askey,
  \emph{Orthogonal Polynomials and Special Functions}.
  Philadelphia, PA: SIAM, 1975.

\bibitem{BEW}
 B. Berndt, R. Evans and K. Williams,
  \emph{Gauss and Jacobi Sums}.
  New York, NY: J. Wiley \& Sons, Inc., 1998

\bibitem{Be93}
  F. Beukers,
  ``Algebraic values of G-functions,''
  \emph{J. Reine Angew. Math.}, vol. 434,
  pp. 45-65, 1993.

\bibitem{Cass96}
  J. W. S. Cassels,
  \emph{Prolegomena to a Middlebrow Arithmetic of Curves of Genus 2}.
  Cambridge, UK: Cambridge University Press, London Mathematical Society Lecture Note Series 230, 1996.

\bibitem{De74}
  P. Deligne,
  ``La conjecture de Weil,''
  \emph{Publ. Math. IHES}, vol. 43,
  pp. 273-307, 1974.

\bibitem{Dw60}
  B. Dwork,
  ``On the rationality of the zeta function of an algebraic variety,''
  \emph{Amer. J. of Math.}, vol. 82,
  pp. 631-648, 1960.

\bibitem{FOP04}
  S. Frechette, K. Ono, and M. Papanikolas,
  ``Gaussian hypergeometric functions and traces of Hecke operators,''
  \emph{Int. Math. Res. Not.},
  pp. 3233-3262, 2004.

\bibitem{Fu69}
  W. Fulton,
  \emph{Algebraic Curves, An Introduction to Algebraic Geometry}.
  New York, NY: W. A. Benjamin, Inc, 1969.

\bibitem{Fu07}
J.G. Fuselier,
``Hypergeometric functions over $\F_{p}$ and relations to elliptic curves and modular forms,''
\emph{Proc. Amer. Math. Soc.}, vol. 138, pp. 109-123, 2010.
%\bibitem{Fu08}
%J.G. Fuselier,
%\emph{Hypergeometric functions $\F_{p}$ and relations to elliptic curves and modular forms},

\bibitem{Gr84}
  J. Greene,
  ``Character sum analogues for hypergeometric and generalized hypergeometric functions over finite fields,''
  Ph.D. dissertation, University of Minnesota, 1984.  

\bibitem{Gr87}
  J. Greene,
  ``Hypergeometric functions over finite fields,''
  \emph{Trans. Amer. Math. Soc.}, vol. 301, 
  pp. 77-101, 1987.
%\bibitem{Hi89}
%  H. Hijikata, A.K. Pizer, and T.R. Shemanske,
%  \emph{The basis problem for modular forms on }$\Gamma_0(N)$,
%  Mem. Amer. Math. Soc. \textbf{82} (1989),
%  vi+159.

\bibitem{IR90}
  K. Ireland and M. Rosen,
  \emph{A Classical Introduction to Modern Number Theory}.
  New York, NY: Springer-Verlag, 1990, 2nd. ed.

\bibitem{Ka90}
 N. M. Katz,
 \emph{Exponential Sums and Differential Equations}.
 Princeton University Press, 1990.

\bibitem{Ko93}
  N. Koblitz,
  \emph{Introduction to Elliptic Curves and Modular Forms}.
  New York, NY: Springer-Verlag, 1993, 2nd. ed.

\bibitem{Ko92}
  M. Koike,
  ``Hypergeometric series over finite fields and Ap\'{e}ry numbers,''
  \emph{Hiroshima Math. J.}, vol. 22,
  pp. 461-467, 1992.

\bibitem{La83}
S. Lang,
\emph{Complex Multiplication}.
New York, NY: Springer-Verlag, 1983.
%\bibitem{La90}
%  S. Lang, 
%  \emph{Cyclotomic Fields I and II},
%  Springer-Verlag, New York, 1990.

\bibitem{Sh05}
  J. S. Milne,
  ``Introduction to Shimura varieties'' in
  \emph{Harmonic analysis, the trace formula and Shimura varieties}.
  Cambridge, MA: Clay Math. Proc., vol. 4,
  pp. 265-378.



\bibitem{Mu85}
D. Mumford,
\emph{Abelian Varieties}.
Bombay, India: Tata Inst. Fund. Research and Oxford University Press, 1985, 2nd. ed.

\bibitem{NZM91}
I. Niven, H.S. Zuckerman and H. L. Montgomery,
\emph{An Introduction to the Theory of Numbers}.
New York, NY: John Wiley \& Sons, Inc., 1991, 5th. ed.

\bibitem{On98}
  K. Ono,
  ``Values of Gaussian hypergeometric series,''
  \emph{Trans. Amer. Math. Soc.}, vol. 350,
  pp. 1205-1223, 1998.

\bibitem{Pa06}
  M. Papanikolas,
  ``A formula and a congruence for Ramanujan's $\tau$-function,''
  \emph{Proc. Amer. Math. Soc.}, vol. 134,
  pp. 333-341, 2006.

\bibitem{Sc87}
  R. Schoof,
  ``Nonsingular plane cubic curves over finite fields,''
  \emph{J. Combin. Theory}, Ser. A vol. 46,
  pp. 183-211, 1987.


\bibitem{Si86}
  J.H. Silverman,
  \emph{The Arithmetic of Elliptic Curves}.
  New York, NY: Springer-Verlag, 1986.

\bibitem{Sl66}
  L. Slater,
  \emph{Generalized Hypergeometric Functions}.
  Cambridge: Cambridge Univ. Press, 1966.

\bibitem{St88}
  P.F. Stiller,
  ``Classical automorphic forms and hypergeometric functions,''
  \emph{J. Number Theory}, vol. 28,
  pp. 219-232, 1988.

%\bibitem{Ve08}
%  M. V. Vega, 
%  ``Hypergeometric functions over finite fields and their relations to algebraic curves'',
%  Ph.D. dissertation,
%  Texas A\&M University, 2009, in preparation.

\bibitem{WaMi}
  W. C. Waterhouse and J. S. Milne,
  ``Abelian varieties over finite fields,''
    \emph{Proc. Symp. Pure Math.}, vol. 20,
   pp. 53-64, 1971.
  
\bibitem{We49}
  A. Weil, 
  ``Number of solutions of equations in finite fields,'' 
  \emph{Bull. AMS}, vol. 55,
  pp. 497-508, 1949.

\bibitem{Xi07}
   C. Xing,
   ``Zeta-functions of curves of genus 2 over finite fields,''
   \emph{Journal of Algebra}, vol. 308,
   pp. 734-741, 2007.
\end{thebibliography}
\end{document}